\documentclass[12pt,english]{amsart}

\input xy
\xyoption{all}

\usepackage{comment}
\usepackage[latin1]{inputenc}
\usepackage{amsfonts,amsmath,amssymb,amsthm}

{\theoremstyle{plain}
\newtheorem{theorem}{Theorem}[section]    % theorem with number

\newtheorem{twisting lemma}[theorem]{Twisting lemma}

\newtheorem{lemma}[theorem]{Lemma}       % lemma with number
\newtheorem{proposition}[theorem]{Proposition}  % lemma with number
\newtheorem{corollary}[theorem]{Corollary}   % lemma with number

}
{\theoremstyle{remark}
\newtheorem{definition}[theorem]{Definition}      % lemma with number
\newtheorem{remark}[theorem]{Remark}   % theorem without number

}

\def\cm{\hbox{\hbox{\rm C}\kern-5pt{\raise 1pt\hbox{$|$}}}}

\def\lhfl#1#2{\smash{\mathop{\hbox to 12mm{\leftarrowfill}}
\limits^{#1}_{#2}}}

\def\rhfl#1#2{\smash{\mathop{\hbox to 12mm{\rightarrowfill}}
\limits^{#1}_{#2}}}

\def\build#1_#2^#3{\mathrel{
\mathop{\kern 0pt#1}\limits_{#2}^{#3}}}

\def\htrait#1#2{\smash{\mathop{\hbox to 12mm{\hrulefill}}
\limits^{#1}_{#2}}}

\def\sxbullet{{\raise 2pt\hbox{\bf .}}}
\begin{document}

\title{Specialization results and ramification conditions}

\author{Fran\c cois Legrand}

\email{Francois.Legrand@math.univ-lille1.fr}

\address{Laboratoire Paul Painlev\'e, Math\'ematiques, Universit\'e Lille 1, 59655 Villeneuve d'Ascq Cedex, France}

\date{\today}

\maketitle

\begin{abstract}
Given a hilbertian field $k$ of characteristic zero and a finite Galois extension $E/k(T)$ with group $G$ such that $E/k$ is regular, we produce some specializations of $E/k(T)$ at points $t_0 \in \mathbb{P}^1(k)$ which have the same Galois group but also specified inertia groups at finitely many given primes. This result has two main applications. Firstly we conjoin it with previous works to obtain Galois extensions of $\mathbb{Q}$ of various finite groups with specified local behavior - ramified or unramified - at finitely many given primes. Secondly, in the case $k$ is a number field, we provide criteria for the extension $E/k(T)$ to satisfy this property: at least one Galois extension $F/k$ of group $G$ is not a specialization of $E/k(T)$. 
\end{abstract}

\section{Introduction}

The {\it{Inverse Galois Problem}} (IGP) over a given number field $k$ asks whether any given finite group $G$ occurs as the Galois group of a finite Galois extension $F/k$. Refined versions of the IGP over $k$ impose some further conditions on the local behavior at finitely many primes of $k$. For example, we may require no prime of a given finite set $\mathcal{S}$ to ramify in $F/k$. From a theorem of Shafarevich, this is always possible if $k= \mathbb{Q}$ and $G$ is solvable \cite[theorem 6.1]{KM04}. Moreover, if $G$ has odd order, one can add the {\it{Grunwald}} conclusion: the completion of $F/\mathbb{Q}$ at each prime $p \in \mathcal{S}$ can be prescribed \cite{Neu79} \cite[(9.5.5)]{NSW08}. Here we are interested in ramification prescriptions at finitely many given primes.

From the {\it{Hilbert Irreducibility Theorem}}, Galois extensions of $k$ of group $G$ can be obtained by specializing Galois extensions $E/k(T)$ with group $G$ such that $E/k$ is {\it{regular}}\footnote{{\it{i.e.}} $E \cap \overline{k} = k$; see \S2.1 for basic terminology.}; many groups occur as the Galois group of such an extension. Let $E/k(T)$ be a Galois extension of group $G$ such that $E/k$ is regular and $\{t_1,\dots,t_r\}$ its branch point set. Our question is whether, for suitable points $t_0 \in \mathbb{P}^1(k) \smallsetminus \{t_1,\dots,t_r\}$, in addition to ${\rm{Gal}}(E_{t_0}/k)=G$, one can prescribe the inertia groups of the specialization $E_{t_0}/k$ of $E/k(T)$ at $t_0$ at finitely many given primes.

Given a prime $\mathcal{P}$ of $k$, not in the finite list of {\it{bad primes for $E/k(T)$}} (definition \ref{bon premier}), and a point $t_0 \in \mathbb{P}^1(k) \smallsetminus \{t_1,\dots,t_r\}$, a classical necessary condition for $\mathcal{P}$ to ramify in $E_{t_0}/k$ is that {\it{$t_0$ meets some branch point $t_{i_\mathcal{P}}$ modulo $\mathcal{P}$}} (definition \ref{rencontre}). As a consequence, $\mathcal{P}$ should admit a prime divisor of residue degree 1 in the field extension $k(t_{i_{\mathcal{P}}})/k$ (say for short that ``$t_{i_{\mathcal{P}}}$ is rationalized by $\mathcal{P}$"). Moreover the inertia group of $E_{t_0}/k$ at $\mathcal{P}$ is known to be generated by some power $g_{i_{\mathcal{P}}}^{a_\mathcal{P}}$ (depending on $t_0$ and $t_{i_\mathcal{P}}$) of the distinguished generator $g_{i_\mathcal{P}}$ of some inertia group of $E\overline{\mathbb{Q}}/\overline{\mathbb{Q}}(T)$ at $t_{i_\mathcal{P}}$. We refer to $\S$2.2 for a precise statement (the ``Specialization Inertia Theorem"), more details and references.

Our main result in $\S$3.1 provides some converse to the latter conclusion: for all primes $\mathcal{P}$ but in a certain finite list $\mathcal{S}_{\rm{exc}}$, if $\mathcal{P}$ rationalizes $t_{i_{\mathcal{P}}}$, in particular if $t_{i_{\mathcal{P}}}$ is itself $k$-rational, then it is possible to prescribe the above exponent $a_{\mathcal{P}}$ for some suitable points $t_0 \in \mathbb{P}^1(k) \smallsetminus \{t_1,\dots,t_r\}$.

Denote the inertia canonical invariant of $E/k(T)$ by $(C_1,\dots,C_r)$, {\it{i.e.}}, for each $i =1,\dots,r$, $C_i$ is the conjugacy class in $G$ of $g_i$ (see $\S$2.1).

\vspace{3mm}

\noindent
{\bf{Theorem 1}} (corollary \ref{Hilbert}) {\it{Let $\mathcal{S}$ be a finite set of primes $\mathcal{P}$ of $k$ not in the finite list $\mathcal{S}_{\rm{exc}}$, each given with a couple $(i_\mathcal{P},a_\mathcal{P})$ where

\vspace{0.5mm}

\noindent
- $i_\mathcal{P}$ is an index in $\{1,\dots,r\}$ such that $t_{i_\mathcal{P}}$ is rationalized by $\mathcal{P}$,

\vspace{0.5mm}

\noindent
- $a_\mathcal{P}$ is a positive integer.

\vspace{1mm}

\noindent
Then there exist infinitely many distinct points $t_0 \in \mathbb{P}^1(k) \smallsetminus \{t_1,\dots,t_r\}$ such that the specialization $E_{t_0}/k$ of $E/k(T)$ at $t_0$ satisfies the following two conditions:

\vspace{0.5mm}

\noindent
{\rm{(1)}} ${\rm{Gal}}(E_{t_0}/k)=G$,

\vspace{0.5mm}

\noindent
{\rm{(2)}} for each prime $\mathcal{P} \in \mathcal{S}$, the inertia group of $E_{t_0}/k$ at $\mathcal{P}$ is generated by some element of $C_{i_\mathcal{P}}^{a_\mathcal{P}}$.}}

\vspace{3mm}

\noindent
Our condition $\mathcal{P} \not \in \mathcal{S}_{\rm{exc}}$ on the primes is that $\mathcal{P}$ should be {\it{a good prime for $E/k(T)$}} such that $t_{i_\mathcal{P}}$ and $1/t_{i_\mathcal{P}}$ are integral over the localization $A_{\mathcal{P}}$ of the integral closure $A$ of $\mathbb{Z}$ in $k$ at $\mathcal{P}$.

Part (2) of the conclusion is proved in a more general situation with the number field $k$ replaced by the quotient field of any Dedekind domain $A$ of characteristic zero and holds for all (but finitely many) points $t_0$ in an arithmetic progression (theorem \ref{Beckmann4}). Furthermore part (1) is satisfied if $k$ is hilbertian or if the inertia canonical invariant of $E/k(T)$ satisfies some {\it{$g$-complete}} hypothesis. We refer to $\S$3.1.2 for more details and extra conclusions on the set of points $t_0$ at which conditions (1) and (2) above simultaneously hold.

Related conclusions can be found in an earlier paper of Plans and Vila \cite{PV05}, for specific finite Galois extensions $E/\mathbb{Q}(T)$ such that $E/\mathbb{Q}$ is regular, generally derived from the rigidity method. Here there are no restriction on the extension $E/\mathbb{Q}(T)$ and the inertia groups may be specified. However a finite list of primes is excluded from our conclusions; in particular any wild ramification situation is left aside.

\vspace{-0.2mm}

Many finite groups are known to occur as the Galois group of a Galois extension $E/\mathbb{Q}(T)$ such that $E/\mathbb{Q}$ is regular (fix $k=\mathbb{Q}$ for simplicity) and with at least one $\mathbb{Q}$-rational branch point (for example, the Monster group does), in which case theorem 1 produces Galois extensions of $\mathbb{Q}$ with the same group which ramify at any finitely many given large enough primes. Some examples are given in $\S$3.2.

\vspace{-0.2mm}

However the assumption on the branch points cannot be removed. Indeed, given an odd prime $p$, Galois extensions of $\mathbb{Q}$ of group $\mathbb{Z}/p\mathbb{Z}$ are known to ramify only at $p$ or at primes $q$ such that $q \equiv 1 \, \,  {\rm{mod}} \, \,  p$ \cite[theorem 1]{Tra90}. And it is known from \cite[corollary 1.3]{DF90} that there are no Galois extension $E/\mathbb{Q}(T)$ of group $\mathbb{Z}/p\mathbb{Z}$ such that $E/\mathbb{Q}$ is regular and with at least one $\mathbb{Q}$-rational branch point.

\vspace{-0.2mm}

On the other hand, theorem 1 also includes trivial ramification at $\mathcal{P}$, by taking $a_\mathcal{P}$ equal to (a multiple of) the order of the elements of $C_{i_\mathcal{P}}$. In this unramified context, similar more precise conclusions are given in the two papers \cite{DG12} and \cite{DG11} of D\`ebes and Ghazi: they have some additional control on the decomposition groups. As shown in $\S$3.3, it is in fact possible to conjoin their statement and theorem 1 to obtain, for any finite group $G$ which occurs as the Galois group of a Galois extension $E/\mathbb{Q}(T)$ with $E/\mathbb{Q}$ regular, {\it{a general existence result of Galois extensions of $\mathbb{Q}$ of group $G$ with specified local behavior (ramified or unramified)}}. Our theorem \ref{DGL} gives the precise statement.

\vspace{-0.2mm}

Given a finite group $G$, we also use theorem 1 in \S4 to give negative answers to the question of whether a given Galois extension $F/k$ of group $G$ is a specialization of a given Galois extension $E/k(T)$ with group $G$ and such that $E/k$ is regular in the case $k$ is a number field. 

\vspace{-0.2mm}

This has been already investigated in \cite{Deb99a}, \cite{DG12}, \cite{DG11} (and also in \cite{DL13} and \cite{DL12} in the non Galois case) for any base field $k$ and positive answers have been given over various fields such as PAC fields, finite fields or complete valued fields. Recall for example that, given a PAC\footnote{\vspace{-0.5mm} {\it{i.e.}} such that any non-empty geometrically irreducible $k$-variety has a Zariski-dense set of $k$-rational points.} field $k$, any Galois extension $F/k$ of group $G$ is the specialization $E_{t_0}/k$ at $t_0$ of any Galois extension $E/k(T)$ with group $G$ and $E/k$ regular for infinitely many distinct points $t_0 \in \mathbb{P}^1(k)$.

However, in the case $k$ is a number field, the situation is more unclear. If our given extension $E/k(T)$ has genus $\geq 2$, the {\it{Faltings Theorem}} shows that a given finite Galois extension $F/k$ of group $G$ occurs as the specialization $E_{t_0}/k$ at $t_0$ for only finitely many distinct points $t_0 \in \mathbb{P}^1(k)$. Moreover there is at least one extension $F/k$ (in fact infinitely many if $G$ is not trivial) for which there is at least one point $t_0$ while, for another one, there may be no point at all: for instance, the imaginary quadratic extension $\mathbb{Q}(i)/\mathbb{Q}$ is not a specialization of $\mathbb{Q}(T)(\sqrt{T^2+1})/\mathbb{Q}(T)$. 

We offer here a systematic approach to produce Galois extensions $E/k(T)$ of group $G$ with $E/k$ regular which are not {\it{$G$-parametric over $k$}}, {\it{i.e.}} such that the answer is negative for at least one Galois extension $F/k$ of group $G$. We refer to $\S$4.1 for some basics on these extensions.

Let $k$ be a number field, $G$ a finite group and $E_1/k(T)$, $E_2/k(T)$ two Galois extensions of group $G$ such that $E_1/k$ and $E_2/k$ each is regular. We use the previous results to produce some specializations of $E_1/k(T)$ of group $G$ which each is not a specialization of $E_2/k(T)$ (and so $E_2/k(T)$ is not $G$-parametric over $k$). More precisely, we provide two different sufficient conditions which each guarantees such a situation. The first one ({\it{Branch Point Hypothesis}}) involves the branch point arithmetic while the second one ({\it{Inertia Hypothesis}}) is a more geometric condition on the inertia of the two extensions $E_1/k(T)$ and $E_2/k(T)$. Theorem \ref{methode} gives the precise statement; it is the aim of $\S$4.2.

These two criteria allow us to give many new examples of non parametric extensions over number fields. We give here (\S4.3)

\noindent
- a result on four branch point Galois extensions (corollary \ref{4 pts}); examples with $G=\mathbb{Z}/2\mathbb{Z}$ and $G= A_5$ are then given (corollaries \ref{Z/2Z}-6),

\noindent
- for each integer $n \geq 3$, a practical result with $G=S_n$ (corollary \ref{Sn gen}).

\noindent
Many other examples are given in \cite{Leg15}, which is specifically devoted to parametric extensions, and in \cite[chapters 2-3]{Leg13c}.

\vspace{3mm}

{\bf{Acknowledgements.}} I am very grateful to my advisor Pierre D\`ebes for his many re-readings and valuable comments. I also wish to thank Michel Emsalem for helpful discussions regarding the Specialization Inertia Theorem and the anonymous referee for his/her thorough work.

\section{First statements on ramification in specializations}

We first set up the terminology and notation for the basic notions we will use in this paper. $\S$2.1 is concerned with Galois extensions of $k(T)$ while we review and complement in $\S$2.2 some general facts about the ramification in their specializations. Finally $\S$2.3 is devoted to a preliminary ramification criterion at one prime.

\subsection{Basics on Galois extensions of $k(T)$}

Given a field $k$ of characteristic zero, fix an algebraic closure $\overline{k}$ of $k$ and denote its absolute Galois group by ${\rm{G}}_k$. Let $T$ be an indeterminate and $E/k(T)$ a finite Galois extension such that $E/k$ is {\it{regular}} ({\it{i.e.}} $E \cap \overline{k}=k$). Denote its Galois group by $G$. For more on below, we refer to \cite[chapter 3]{Deb09}. 

\subsubsection{Branch points}

Denote the integral closure of $\overline{k}[T]$ (resp. of $\overline{k}[1/T]$) in $E\overline{k}$ by $\overline{B}$ (resp. by $\overline{B^*}$). A point $t_0 \in \overline{k}$ (resp. $\infty$) is said to be {\it{a branch point of $E/k(T)$}} if the prime $(T-t_0) \, \overline{k}[T]$ (resp. $(1/T) \, \overline{k}[1/T]$) ramifies in $\overline{B}$ (resp. in $\overline{B^*}$). The extension $E/k(T)$ has only finitely many branch points, denoted by $t_1,\dots,t_r$.

\subsubsection{Inertia canonical invariant} Fix a {\it{coherent system $\{\zeta_n\}_{n=1}^\infty$ of roots of unity}}, {\it{i.e.}} $\zeta_n$ is a primitive $n$-th root of unity and $\zeta_{nm}^n=\zeta_m$ for any positive integers $n$ and $m$.

To each $t_i$ can be associated a conjugacy class $C_i$ of $G$, called the {\it{inertia canonical conjugacy class (associated with $t_i$)}}, in the following way. The inertia groups of $E\overline{k}/\overline{k}(T)$  at $t_i$ are cyclic conjugate groups of order equal to the ramification index $e_i$. Furthermore each of them has a distinguished generator corresponding to the automorphism $(T-t_i)^{1/e_i} \mapsto \zeta_{e_i} (T-t_i)^{1/e_i}$ of $\overline{k}(((T-t_i)^{1/e_i}))$ (replace $T-t_i$ by $1/T$ if $t_i=\infty$). Then $C_i$ is the conjugacy class of all the distinguished generators of the inertia groups at $t_i$. The unordered $r$-tuple $(C_1,\dots,C_r)$ is called {\it{the inertia canonical invariant of $E/k(T)$}}.

\subsubsection{Specializations}

If $t_0 \in \mathbb{P}^1(k)$ is not a branch point, the residue field of some prime above $t_0$ in $E/k(T)$ is denoted by $E_{t_0}$ and we call the extension $E_{t_0}/k$ {\it{the specialization of $E/k(T)$ at $t_0$}} (this does not depend on the choice of the prime above $t_0$ since the extension $E/k(T)$ is Galois). It is a Galois extension of $k$ of Galois group a subgroup of $G$, namely the decomposition group of the extension $E/k(T)$ at $t_0$.

\vspace{2mm}

In the case $E/k(T)$ is given by a polynomial $P(T,X) \in k[T][X]$, the following lemma is useful:

\begin{lemma} \label{spec}
Let $P(T,X) \in k[T][X]$ be a monic (with respect to $X$) separable polynomial of splitting field $E$ over $k(T)$. Then, for any $t_0 \in k$ such that the specialized polynomial $P(t_0,X)$ is separable over $k$, $t_0$ is not a branch point and the specialization $E_{t_0}/k$ of $E/k(T)$ at $t_0$ is the splitting extension over $k$ of $P(t_0,X)$.
\end{lemma}

\subsection{Conditions on ramification in specializations}

The aim of this subsection is the ``Specialization Inertia Theorem" ($\S$2.2.3) which is a slightly more general version of a Beckmann's result \cite[proposition 4.2]{Bec91}. {\hbox{We before review and complement in $\S$2.2.1-2 some background.}}

Let $A$ be a Dedekind domain of characteristic zero, $k$ its quotient field and $\mathcal{P}$ a (non-zero) prime of $A$. Denote the valuation of $k$ corresponding to $\mathcal{P}$ by $v_\mathcal{P}$.

\subsubsection{Meeting} Throughout this subsubsection, we will identify $\mathbb{P}^1(k)$ with $k \cup \{\infty\}$ and set

\noindent
- $1/\infty = 0$,

\noindent
- $1 / 0 = \infty$,

\noindent
- $v_\mathcal{P}(\infty) = -\infty$,

\noindent
- $v_\mathcal{P}(0) = \infty$.

\vspace{2.5mm}

Now recall the following definition:

\begin{definition} \label{rencontre} 
(1) Let $F/k$ be a finite extension, $A_F$ the integral closure of $A$ in $F$, $\mathcal{P}_F$ a (non-zero) prime of $A_F$ and $t_0$, $t_1 \in \mathbb{P}^1(F)$. We say that {\it{$t_0$ and $t_1$ meet modulo $\mathcal{P}_F$}} if either one of the following two conditions holds: 

\vspace{0.5mm}

(a) $v_{\mathcal{P}_F}(t_0) \geq 0$, $v_{\mathcal{P}_F}(t_1) \geq 0$ and $v_{\mathcal{P}_F}(t_0-t_1) > 0$,

\vspace{0.5mm}

(b) $v_{\mathcal{P}_F}(t_0) \leq 0$, $v_{\mathcal{P}_F}(t_1) \leq 0$ and $v_{\mathcal{P}_F}((1/t_0) - (1/t_1)) > 0$.

\vspace{1mm}

\noindent
(2) Given $t_0$, $t_1$ $\in \mathbb{P}^1(\overline{k})$, we say that {\it{$t_0$ and $t_1$ meet modulo $\mathcal{P}$}} if there exists some finite extension $F/k$ satisfying the following two conditions:

\vspace{0.5mm}

(a) $t_0$, $t_1$ $\in \mathbb{P}^1(F)$,

\vspace{0.5mm}

(b) $t_0$ and $t_1$ meet modulo some prime of  $F$ lying over $\mathcal{P}$.
\end{definition}

\begin{remark} \label{3.2}
(1) Part (2) of definition \ref{rencontre} does not depend on the choice of the finite extension $F/k$ such that $t_0$, $t_1$ $\in \mathbb{P}^1(F)$.

\vspace{0.5mm}

\noindent
(2) If $t_0 \in \mathbb{P}^1(k)$ meets $t_1$ modulo $\mathcal{P}$, then $t_0$ meets each $k$-conjugate of $t_1$ modulo $\mathcal{P}$. 
\end{remark}

Let $T$ be an indeterminate. Throughout this paper, the irreducible polynomial over $k$ of any point $t_1 \in \mathbb{P}^1(\overline{k})$ will be denoted by $m_{t_1}(T)$ (set $m_{t_1}(T)=1$ if $t_1 = \infty$). Denote its constant coefficient by $a_{t_1}$. Then the irreducible polynomial of $1/t_1$ over $k$ is 

\vspace{0.25mm}

\noindent
- $m_{1/t_1}(T)=(1/a_{t_1})\, T^{{\rm{deg}}(m_{t_1}(T))} \, m_{t_1}(1/T)$ if $t_1 \in \overline{k} \smallsetminus \{0\}$,

\vspace{0.5mm}

\noindent
- $m_{1/t_1}(T)=1$ if $t_1=0$,

\vspace{0.5mm}

\noindent
- $m_{1/t_1}(T)=T$ if $t_1 = \infty$.

\vspace{2.5mm}

Fix $t_1 \in \mathbb{P}^1(\overline{k})$. Throughout $\S$2.2.1, we will assume that $v_\mathcal{P}(a_{t_1})=0$ if $t_1 \not=0$ to make the intersection multiplicity well-defined in definition \ref{**} below. Let $t_0 \in \mathbb{P}^1(k)$.

\begin{definition} \label{**}
{\it{The intersection multiplicity $I_{\mathcal{P}}(t_0,t_1)$ of $t_0$ and $t_1$ at $\mathcal{P}$}} is 
$I_\mathcal{P}(t_0,t_1)= \left \{ \begin{array} {ccc}
          v_{\mathcal{P}}(m_{t_1}(t_0)) & {\rm{if}} & v_\mathcal{P}(t_0) \geq 0, \\
          v_{\mathcal{P}}(m_{1/t_1}(1/t_0)) & {\rm{if}} &  v_\mathcal{P}(t_0) \leq 0. \\   
          \end{array} \right.$
\end{definition}

Lemma \ref{lemme} below will be used on several occasions in this paper:

\begin{lemma} \label{lemme} 
{\rm{(1)}} If $I_{\mathcal{P}}(t_0,t_1) >0$, then $t_0$ and $t_1$ meet modulo $\mathcal{P}$.

\vspace{0.5mm}

\noindent
{\rm{(2)}} The converse is true if $m_{t_1}(T) \in A_\mathcal{P}[T]$.
\end{lemma}

\begin{proof}
First of all, we note the following simple statement which will be used on several occasions in this paper:

\vspace{1.5mm}

\noindent
($*$) {\it{Let $m(T) \in A_\mathcal{P}[T]$ be a non constant monic polynomial, $L/k$ any finite extension, $\mathcal{Q}$ a prime of $L$ above $\mathcal{P}$ and $t \in L$ such that $v_\mathcal{Q}(m(t)) \geq 0$ (in particular if $t$ is a root of $m(T)$). Then $v_\mathcal{Q}(t) \geq 0$.}}

\vspace{1.5mm}

\noindent
Indeed assume that $v_\mathcal{Q}(t) <0$. Set $m(T)=a_0 + a_1 T + \dots + a_{n-1} T^{n-1} + T^n$. Since $m(T) \in A_\mathcal{P}[T]$, one has $v_\mathcal{Q}(a_j t^j) > v_\mathcal{Q}(t^n)$ for each index $j \in \{0,\dots,n-1\}$. Hence $v_\mathcal{Q}(m(t)) = v_\mathcal{Q}(t^n) <0$; a contradiction.

\vspace{2mm}

To prove lemma \ref{lemme}, set $m_{t_1}(T)= \prod_{i=1}^n (T-t_i)$ (if $t_1 \not = \infty$) and fix a prime $\mathcal{Q}$ of $k(t_1,\dots,t_n)$ above $\mathcal{P}$. We successively prove conclusions (1) and (2).

\vspace{2mm}

\noindent
(1) First assume that $v_{\mathcal{P}}(t_0) \geq 0$. Then one has $v_{\mathcal{P}}(m_{t_1}(t_0)) > 0$ from our assumption $I_\mathcal{P}(t_0,t_1) > 0$ and $t_1 \not= \infty$ (otherwise $1=m_{t_1}(t_0) \in \mathcal{P}A_\mathcal{P}$). Hence one has $\sum_{i=1}^n v_{\mathcal{Q}}(t_0-t_i) >0$. Consequently there exists some index $i_0 \in \{1,\dots,n\}$ such that $v_{\mathcal{Q}}(t_0-t_{i_0}) >0$. Since $v_{\mathcal{Q}}(t_0) \geq 0$, one has $v_\mathcal{Q}(t_{i_0}) \geq 0$. Hence $t_0$ and $t_{i_0}$ meet modulo $\mathcal{P}$ and the conclusion follows from part (2) of remark \ref{3.2}.

Now assume that $v_{\mathcal{P}}(t_0) \leq 0$. Then one has $v_{\mathcal{P}}(m_{1/t_1}(1/t_0)) > 0$ and $t_1 \not=0$ (otherwise $1=m_{1/t_1}(1/t_0) \in \mathcal{P}A_\mathcal{P}$). If $t_1=\infty$, then $t_0$ and $t_1$ meet modulo $\mathcal{P}$. If $t_1 \not=\infty$, one has $m_{1/t_1}(T) = \prod_{i=1}^n (T-(1/t_i))$. Hence $\sum_{i=1}^n v_{\mathcal{Q}}((1/t_0)- (1/t_i)) >0$. Consequently there exists some index $i_0 \in \{1,\dots,n\}$ such that $v_{\mathcal{Q}}((1/t_0)-(1/t_{i_0})) >0$. As before, $t_0$ and $t_{i_0}$ meet modulo $\mathcal{P}$ and one concludes from part (2) of remark \ref{3.2}.

\vspace{2mm}

\noindent
(2) Now assume that $t_0$ and $t_1$ meet modulo $\mathcal{P}$ and $m_{t_1}(T) \in A_\mathcal{P}[T]$ \footnote{and so $m_{1/t_1}(T)$ does too due to our assumption stated before definition \ref{**}.}. It is easily checked that conclusion (2) holds if $t_1 \in \{0,\infty\}$, so assume that $t_1 \not \in \{0,\infty\}$. 

First consider the case $v_\mathcal{Q}(t_0) \geq 0$, $v_\mathcal{Q}(t_1) \geq 0$ and $v_\mathcal{Q}(t_0-t_1) > 0$. Given an index $i \in \{1,\dots,n\}$, statement ($*$) (applied to the polynomial $m_{t_1}(T)$) shows that one has $v_\mathcal{Q}(t_i) \geq 0$, and then $v_\mathcal{Q}(t_0 - t_i) \geq 0$. Hence $v_\mathcal{Q}(m_{t_1}(t_0)) \geq v_\mathcal{Q}(t_0 - t_1) >0$, {\it{i.e.}} $I_\mathcal{P}(t_0,t_1) >0$.

Now consider the case $v_\mathcal{Q}(t_0) \leq 0$, $v_\mathcal{Q}(t_1) \leq 0$ and $v_\mathcal{Q}((1/t_0) - (1/t_1)) > 0$. Given an index $i \in \{1,\dots,n\}$, statement ($*$) (applied this time to the polynomial $m_{1/t_1}(T)$) shows that one has $v_\mathcal{Q}(1/t_i) \geq 0$, and then $v_\mathcal{Q}((1/t_0) - (1/t_i)) \geq 0$. Hence $v_\mathcal{Q}(m_{1/t_1}(1/t_0)) \geq v_\mathcal{Q}((1/t_0) - (1/t_1)) >0$, {\it{i.e.}} $I_\mathcal{P}(t_0,t_1) >0$.
\end{proof}

\subsubsection{Good primes} Continue with the same notation as before. Let $G$ be a finite group and $E/k(T)$ a Galois extension of group $G$ with $E/k$ regular. Denote its branch point set by $\{t_1,\dots,t_r\}$.

\begin{definition} \label{bon premier}
We say that $\mathcal{P}$ is {\it{a bad prime for $E/k(T)$}} if at least one of the following four conditions holds:

\vspace{0.5mm}

\noindent
(1) $|G| \in \mathcal{P}$,

\vspace{0.5mm}

\noindent
(2) two distinct branch points meet modulo $\mathcal{P}$,

\vspace{0.5mm}

\noindent
(3) $E/k(T)$ has {\it{vertical ramification}} at $\mathcal{P}$, {\it{i.e.}} the prime $\mathcal{P} A[T]$ of $A[T]$ ramifies in the integral closure of $A[T]$ in $E
$ \footnote{If $G$ has trivial center, this condition may be removed \cite[proposition 2.3]{Bec91}.},

\vspace{0.5mm}

\noindent
(4) $\mathcal{P}$ ramifies in $k(t_1,\dots,t_r)/k$.

\vspace{0.5mm}

\noindent
Otherwise $\mathcal{P}$ is called {\it{a good prime for $E/k(T)$}}.
\end{definition}

\begin{remark} \label{condd}
(1) There are only finitely many bad primes for $E/k(T)$.

\vspace{1mm}

\noindent
(2) Condition (4) above does not appear in \cite{Bec91} but seems to be missing for the proof of proposition 4.2 of this paper to work. Indeed, although it is stated and used at the beginning of the proof there, it seems unclear that any prime of $A$ which ramifies in the extension $k(t_1,\dots,t_r)/k$ should be a bad prime for $E/k(T)$. 

In fact, if $\mathcal{P}$ satisfies condition (4) and this extra condition:

\vspace{1mm}

\noindent
(4') {\it{$t_i$ or $1/t_i$ is integral over $A_\mathcal{P}$ (i.e. $m_{t_i}(T)$ or $m_{1/t_i}(T)$ has coefficients in $A_\mathcal{P}$) for each non $k$-rational branch point $t_i$,}}

\vspace{1mm}

\noindent
then $\mathcal{P}$ satisfies condition (2) of definition \ref{bon premier} \footnote{and then $\mathcal{P}$ is a bad prime in the sense of Beckmann.}.

Indeed, if $\mathcal{P}$ ramifies in $k(t_1,\dots,t_r)/k$, then $\mathcal{P}$ does in some $k(t_i)/k$ and so $t_i$ is not $k$-rational. So assume from the extra condition (4') that $t_i$ is integral over $A_\mathcal{P}$ (the other case for which $1/t_i$ is integral over $A_\mathcal{P}$ is quite similar). Hence $\mathcal{P}A_\mathcal{P}$ contains the discriminant of the integral $k$-basis $\{1,t_i,\dots,t_i^{[k(t_i)\, : \,k]-1}\}$ of $k(t_i)$, {\it{i.e.}} the discriminant of $m_{t_i}(T)$. 

This sole condition shows that condition (2) of definition \ref{bon premier} holds. Indeed, first remark that $t_i$ is not $k$-rational (otherwise $1 \in \mathcal{P}A_\mathcal{P}$). Let $\mathcal{Q}$ be a prime of the splitting field over $k$ of $m_{t_i}(T)=\prod_j(T-t_j)$ above $\mathcal{P}$. As $\prod_{j \not=j'} (t_j-t_{j'}) \in \mathcal{P} A_\mathcal{P}$, there are two indices $j \not=j'$ such that $v_\mathcal{Q}(t_j-t_{j'}) >0$. If $v_\mathcal{Q}(t_j) \geq 0$, then $v_\mathcal{Q}(t_{j'}) \geq 0$ and $t_j$ and $t_{j'}$ meet modulo $\mathcal{P}$. If $v_\mathcal{Q}(t_j) < 0$, then $v_\mathcal{Q}(t_{j'}) < 0$ and $v_\mathcal{Q}((1/t_j)-(1/t_{j'})) = v_\mathcal{Q}(t_j-t_{j'}) - v_\mathcal{Q}(t_j) - v_\mathcal{Q}(t_{j'}) >0$. Hence $t_j$ and $t_{j'}$ meet modulo $\mathcal{P}$.
\end{remark}

In particular, we obtain lemma \ref{derivee} below:

\begin{lemma} \label{derivee}
Let $i \in \{1,\dots,r\}$ and $t_0 \in A_\mathcal{P}$. Assume that $m_{t_i}(T) \in A_{\mathcal{P}}[T]$, $v_{\mathcal{P}}(m_{t_i}(t_0)) > 0$ and $v_{\mathcal{P}}(m'_{t_i}(t_0)) > 0$. Then $\mathcal{P}$ is a bad prime for $E/k(T)$.
\end{lemma}

\subsubsection{Ramification in specializations of $E/k(T)$} Continue with the same notation as before. For each index $i \in \{1,\dots,r\}$, let $g_i$ be the distinguished generator of some inertia group of $E\overline{k}/\overline{k}(T)$ at $t_i$. 

\vspace{3mm}

\noindent
{\bf{Specialization Inertia Theorem.}} {\it{Let $t_0 \in \mathbb{P}^1(k) \smallsetminus \{t_1,\dots,t_r\}$.

\vspace{1mm}

\noindent
{\rm{(1)}} If $\mathcal{P}$ ramifies in $E_{t_0}/k$, then $E/k(T)$ has vertical ramification at $\mathcal{P}$ or $t_0$ meets some branch point modulo $\mathcal{P}$.

\vspace{1mm}

\noindent
{\rm{(2)}} Fix an index $j \in \{1,\dots,r\}$ such that $t_0$ and $t_j$ meet modulo $\mathcal{P}$. Assume that the following two conditions hold:

\vspace{0.5mm}

{\rm{(a)}} $\mathcal{P}$ is a good prime for $E/k(T)$,

\vspace{0.5mm}

{\rm{(b)}} $t_j$ and $1/t_j$ are integral over $A_\mathcal{P}$.

\vspace{0.25mm}

\noindent
Then the inertia group of $E_{t_0}/k$ at $\mathcal{P}$ is (conjugate in $G$ to) $\langle g_j^{I_\mathcal{P}(t_0,t_j)} \rangle$.}}

\vspace{3mm}

In the case $t_j \not \in \{0,\infty\}$, condition (b) in part (2) above is equivalent to $t_j$ being a unit in $\overline{k}$ with respect to any prolongation of $v_\mathcal{P}$ to $\overline{k}$ (statement ($*$)). It will be used on several occasions in this paper; we will say for short that ``$\mathcal{P}$ unitizes $t_j$".

As already alluded to, this result is a version of \cite[proposition 4.2]{Bec91} with less restrictive hypotheses. Part (1) may be obtained from the algebraic cover theory of Grothendieck while part (2) follows from the original proof of \cite[proposition 4.2]{Bec91} and some work of Flon \cite[theorem 1.3.3]{Flo02} (and the necessary adjustment alluded to in part (2) of remark \ref{condd}). ``A unified and detailed proof" is given in \cite{Leg13c}.

\subsection{Ramification criteria at one prime} Our next goal (achieved with theorem \ref{Beckmann4}) is to show that, for some good choice of the specialization point $t_0$, ramification can be prescribed at finitely many primes in the specialization $E_{t_0}/k$ within the Specialization Inertia Theorem limitations. We start by the special but useful case where there is a single prime and the requirement is that it does ramify (corollary \ref{transition2}).

Continue with the same notation as before. Let $x_\mathcal{P}$ be a generator of the maximal ideal $\mathcal{P} A_\mathcal{P}$ of $A_\mathcal{P}$. Assume in proposition \ref{transition} below that $\mathcal{P}$ is a good prime for $E/k(T)$ unitizing each branch point.

\begin{proposition} \label{transition}
Let $t_0 \in \mathbb{P}^1(k) \smallsetminus \{t_1,\dots,t_r\}$ such that $v_\mathcal{P}(t_0) \geq 0$ (resp. $v_\mathcal{P}(t_0) \leq 0$) and neither $t_0$ nor $t_0 + x_\mathcal{P}$ is in $\{t_1,\dots,t_r\}$ (resp. neither $t_0$ nor $t_0/(1+x_\mathcal{P} \, t_0)$ \footnote{Replace $t_0/(1+x_\mathcal{P} \, t_0)$ by $1/x_\mathcal{P}$ if $t_0=\infty$.} is in $\{t_1,\dots,t_r\}$). Then the following two conditions are equivalent:

\vspace{0.5mm}

\noindent
{\rm{(1)}} $t_0$ meets some branch point modulo $\mathcal{P}$ (in both cases),

\vspace{0.5mm}

\noindent
{\rm{(2)}} $\mathcal{P}$ ramifies in $E_{t_0}/k$ or in $E_{t_0 + x_\mathcal{P}}/k$ (resp. in $E_{t_0}/k$ or in $E_{t_0/(1+x_\mathcal{P} \, t_0)}/k$).
\end{proposition}

\begin{proof} 
We may assume that $v_\mathcal{P}(t_0) \geq 0$ (the other case for which $v_\mathcal{P}(t_0) \leq 0$ is quite similar). 

First assume that condition (2) holds. From part (1) of the Specialization Inertia Theorem, one may assume that $\mathcal{P}$ ramifies in $E_{t_0 + x_\mathcal{P}}/k$. Hence $t_0 + x_\mathcal{P}$ meets some branch point $t_i$ modulo $\mathcal{P}$. Since $m_{t_i}(T) \in A_\mathcal{P}[T]$, the converse in part (1) of lemma \ref{lemme} holds and $I_{\mathcal{P}}(t_0+x_\mathcal{P},t_i)>0$, {\it{i.e.}} $v_{\mathcal{P}}(m_{t_i}(t_0 + x_\mathcal{P})) >0$. From Taylor's formula, there exists some element $R_\mathcal{P} \in A_\mathcal{P}$ such that 
$$m_{t_i}(t_0) = m_{t_i}(t_0 + x_\mathcal{P}) + x_\mathcal{P} \, R_\mathcal{P}$$
Hence $v_\mathcal{P}(m_{t_i}(t_0)) >0$, {\it{i.e.}} $I_\mathcal{P}(t_0,t_i) >0$. It then remains to apply part (1) of lemma \ref{lemme} to finish the proof of implication (2) $\Rightarrow$ (1).

Now assume that $t_0$ and $t_i$ meet modulo $\mathcal{P}$ (and then $I_\mathcal{P}(t_0,t_i) >0$ from the converse in part (1) of lemma \ref{lemme}). From part (2) of the Specialization Inertia Theorem, $\mathcal{P}$ ramifies in $E_{t_0}/k$ if and only if $I_\mathcal{P}(t_0,t_i)$ is not a multiple of the order of the distinguished generator $g_i$, {\it{i.e.}} if and only if $v_\mathcal{P}(m_{t_i}(t_0))$ is not either. Hence we may assume that $v_\mathcal{P}(m_{t_i}(t_0)) \geq 2$. Taylor's formula yields
$$m_{t_i}(t_0 + x_\mathcal{P}) = m_{t_i}(t_0) + x_\mathcal{P} \, m_{t_i}'(t_0) + x_{\mathcal{P}}^2 \, R_\mathcal{P}$$
with $R_\mathcal{P} \in A_{\mathcal{P}}$. Then $v_{\mathcal{P}}(m_{t_i}(t_0+x_\mathcal{P}))=1$ since $v_{\mathcal{P}}(m_{t_i}(t_0)) \geq 2$, $v_{\mathcal{P}}(x_\mathcal{P} \, m_{t_i}'(t_0))=1$ (lemma \ref{derivee}) and $v_{\mathcal{P}}(x_\mathcal{P}^2 \, R_{\mathcal{P}}) \geq 2$. Hence $\mathcal{P}$ ramifies in $E_{t_0 + x_\mathcal{P}}/k$ and condition (2) holds.
\end{proof}

Now recall the following definition:

\begin{definition} \label{diviseur}
Let $P(T) \in k[T]$ be a non constant polynomial. We say that $\mathcal{P}$ is {\it{a prime divisor of $P(T)$}} if there exists some element $t_0 \in k$ such that $v_{\mathcal{P}}(P(t_0)) >0$.
\end{definition}

\begin{remark} \label{rem div}
Assume that $P(T)$ is in $A_\mathcal{P}[T]$ and $v_\mathcal{P}(P(t_0)) >0$. Fix $a \in \mathcal{P} A_\mathcal{P}$. As noted in the second paragraph of the proof of proposition \ref{transition}, one has $v_{\mathcal{P}}(P(t_0 + a)) >0$. Moreover, if $v_\mathcal{P}(a) > v_\mathcal{P}(P(t_0))$, then $v_\mathcal{P}(P(t_0+a))=v_\mathcal{P}(P(t_0))$.
\end{remark}

Set $m_{\bf{\underline{t}}}(T) = \prod_{i=1}^r m_{t_i}(T)$ and $m_{1/\bf{\underline{t}}}(T) = \prod_{i=1}^r m_{1/t_i}(T)$. Then corollary \ref{transition2} below follows:

\begin{corollary} \label{transition2}
Assume that $\mathcal{P}$ is a good prime for $E/k(T)$ unitizing each branch point. Then the following two conditions are equivalent:

\vspace{0.5mm}

\noindent
{\rm{(1)}} $\mathcal{P}$ ramifies in at least one specialization of $E/k(T)$,

\vspace{0.5mm}

\noindent
{\rm{(2)}} $\mathcal{P}$ is a prime divisor of $m_{\bf{\underline{t}}}(T) \cdot m_{1/\bf{\underline{t}}}(T)$.
\end{corollary}

\begin{proof}
First assume that there is some $t_0 \in \mathbb{P}^1(k) \smallsetminus \{t_1,\dots,t_r\}$ such that $\mathcal{P}$ ramifies in $E_{t_0}/k$. Suppose that $v_\mathcal{P}(t_0) \geq 0$ (the other case for which $v_\mathcal{P}(t_0) \leq 0$ is similar). As noted in the second paragraph of the proof of proposition \ref{transition}, one has $v_\mathcal{P}(m_{t_i}(t_0)) >0$ for some $i \in \{1,\dots,r\}$. But $m_{t_1}(T), \dots, m_{t_r}(T), m_{1/t_1}(T), \dots, m_{1/t_r}(T) \in A_\mathcal{P}[T]$ and $t_0 \in A_\mathcal{P}$. Then $v_\mathcal{P}(m_{\bf{\underline{t}}}(t_0) \cdot m_{1/\bf{\underline{t}}}(t_0)) >0$ and condition (2) holds.

Conversely assume that condition (2) holds. Fix $t_0 \in k$ such that $v_\mathcal{P}(m_{\bf{\underline{t}}}(t_0) \cdot m_{1/\bf{\underline{t}}}(t_0)) >0$. From statement ($*$), one has $v_\mathcal{P}(t_0) \geq 0$. Assume that $v_\mathcal{P}(m_{\bf{\underline{t}}}(t_0)) >0$ (the other case for which $v_\mathcal{P}(m_{1/\bf{\underline{t}}}(t_0)) >0$ is similar). Then there is an index $i \in \{1,\dots,r\}$ such that $v_\mathcal{P}(m_{t_i}(t_0)) >0$ (and so condition (1) of proposition \ref{transition} holds from part (1) of lemma \ref{lemme}). From remark \ref{rem div}, one may assume that neither $t_0$ nor $t_0 + x_\mathcal{P}$ is in $\{t_1,\dots,t_r\}$ and the conclusion follows from proposition \ref{transition}.
\end{proof}

\section{Specializations with specified local behavior}

This section is devoted to theorem \ref{Beckmann4} (our most general result) which is more general than theorem 1 from the introduction; it is the aim of $\S$3.1.1. We then give in $\S$3.1.2 two more practical forms of this statement (corollaries \ref{Hilbert} and \ref{gc}). We next apply these results to some classical Galois extensions of $\mathbb{Q}(T)$ in $\S$3.2. Finally $\S$3.3 is devoted to theorem \ref{DGL} which, as alluded to in the introduction, conjoins theorem \ref{Beckmann4} and previous results.

\subsection{Specializations with specified inertia groups} Fix a Dedekind domain $A$ of characteristic zero and denote its quotient field by $k$. Let $G$ be a finite group and $E/k(T)$ a Galois extension of group $G$ such that $E/k$ is regular. Denote its branch point set by $\{t_1, \dots, t_r\}$ and its inertia canonical invariant by $(C_1,\dots,C_r)$.

\subsubsection{General result} Given a positive integer $s$, fix $s$ distinct good primes $\mathcal{P}_1,\dots,\mathcal{P}_s$ for $E/k(T)$ and $s$ couples $(i_1,a_1), \dots, (i_s,a_s)$ where, for each index $j \in \{1,\dots,s\}$,

\vspace{0.5mm}

\noindent
(a) $i_j$ is an index in $\{1,\dots,r\}$ such that $\mathcal{P}_j$ is a prime divisor of the polynomial $m_{t_{i_j}}(T) \cdot m_{1/t_{i_j}}(T)$ and unitizes $t_{i_j}$,

\vspace{0.25mm}

\noindent
(b) $a_j$ is a positive integer.

\begin{theorem} \label{Beckmann4}
There are infinitely many distinct $t_0 \in k \smallsetminus \{t_1,\dots,t_r\}$ such that, for each $j \in \{1,\dots,s\}$, the inertia group at $\mathcal{P}_j$ of the specialization $E_{t_0}/k$ of $E/k(T)$ at $t_0$ is generated by some element of $C_{i_j}^{a_j}$.
\end{theorem}

\noindent
{\it{Addendum}} \ref{Beckmann4}.  For each $j \in \{1,\dots,s\}$, let $x_{\mathcal{P}_j} \in A$ be a generator of $\mathcal{P}_j A_{\mathcal{P}_j}$. Denote the set of all $j \in \{1,\dots,s\}$ such that $t_{i_j} \not= \infty$ by $S$. 

There exists some $\theta \in k$ such that the conclusion of theorem \ref{Beckmann4} holds at any point $t_{0,u} \in k \smallsetminus \{t_1,\dots,t_r\}$ of the form $t_{0,u}= \theta + u \, \prod_{l \in {S}} x_{\mathcal{P}_l}^{a_l+1}$, with $u$ any element of $k$ such that $v_{\mathcal{P}_l}(u) \geq 0$ for each index $l \in \{1,\dots,s\}$. Furthermore, if $S= \{1,\dots,s\}$ (in particular if $\infty$ is not a branch point), then such an element $\theta$ may be chosen in $A$.

\begin{remark} \label{ur}
For some $j$, there may be no index $i$ such that $\mathcal{P}_j$ is a prime divisor of $m_{t_{i}}(T) \cdot m_{1/t_{i}}(T)$. In this case, if $\mathcal{P}_j$ unitizes each branch point, then $E_{t_0}/k$ ramifies at $\mathcal{P}_j$ for no point $t_0 \in \mathbb{P}^1(k) \smallsetminus \{t_1,\dots,t_r\}$ (corollary \ref{transition2}). If there exists such an index $i_j$, theorem \ref{Beckmann4} also provides specializations of $E/k(T)$ which each does not ramify at $\mathcal{P}_j$, by taking $a_j$ equal to (a multiple of) the order of the elements of $C_{i_j}$. Conjoining these two facts yields the following:

\vspace{1.5mm}

\noindent
{\it{Assume that each prime $\mathcal{P}_j$, $j=1,\dots,s$, is a good prime for $E/k(T)$ unitizing each branch point. Then there exist infinitely many distinct points $t_0 \in k \smallsetminus \{t_1,\dots,t_r\}$ such that $\mathcal{P}_1,\dots,\mathcal{P}_s$ are unramified in $E_{t_0}/k$.}}

\vspace{1.5mm}

\noindent
As in theorem \ref{Beckmann4}, the conclusion holds at all (but finitely many) points in an arithmetic progression.
\end{remark}

Theorem \ref{Beckmann4} is proved in $\S$3.4.

\subsubsection{Conjoining theorem \ref{Beckmann4} and the Hilbert specialization property} Continue with the notation from $\S$3.1.1. We give two situations where infinitely many specializations from theorem \ref{Beckmann4} have Galois group $G$.

\vspace{2.5mm}

\noindent
(a) {\it{Hilbertian base field.}} Assume that $k$ is hilbertian and fix an element $\theta$ as in addendum \ref{Beckmann4}. From \cite[lemma 3.4]{Gey78}, there exist infinitely many distinct elements $u \in \bigcap_{l=1}^s A_{\mathcal{P}_l}$ such that the specializations $E_{t_{0,u}}/k$ of $E/k(T)$ at $t_{0,u}=\theta + u \, \prod_{l \in {S}} x_{\mathcal{P}_l}^{a_l+1}$ are linearly disjoint and each has Galois group $G$. Hence corollary \ref{Hilbert} below follows:

\begin{corollary} \label{Hilbert}
For infinitely many distinct points $t_0 \in k \smallsetminus \{t_1,\dots,t_r\}$ in some arithmetic progression, the specializations $E_{t_0}/k$ of $E/k(T)$ at $t_0$ are linearly disjoint and each satisfies the following two conditions:

\vspace{0.5mm}

\noindent
{\rm{(1)}} ${\rm{Gal}}(E_{t_{0}}/k)=G$,

\vspace{0.5mm}

\noindent
{\rm{(2)}} for each index $j \in \{1,\dots,s\}$, the inertia group of $E_{t_{0}}/k$ at $\mathcal{P}_j$ is generated by some element of $C_{i_j}^{a_j}$.
\end{corollary}

\noindent
(b) {\it{$g$-complete hypothesis.}} Recall that a set $\Sigma$ of conjugacy classes of $G$ is said to be {\it{$g$-complete}} (a terminology due to Fried \cite{Fri95}) if no proper subgroup of $G$ intersects each conjugacy class in $\Sigma$. For instance, the set of all conjugacy classes of $G$ is g-complete \cite{Jor72}.

Assume in corollary \ref{gc} below that $k$ is a number field and the set $\{C_1,\dots,C_r\}$ is g-complete.

\begin{corollary} \label{gc}
For any point $t_0 \in k \smallsetminus \{t_1,\dots,t_r\}$ in some arithmetic progression, the specialization $E_{t_0}/k$ of $E/k(T)$ at $t_0$ satisfies the following two conditions:

\vspace{0.5mm}

\noindent
{\rm{(1)}} ${\rm{Gal}}(E_{t_{0}}/k)=G$,

\vspace{0.5mm}

\noindent
{\rm{(2)}} for each index $j \in \{1,\dots,s\}$, the inertia group of $E_{t_{0}}/k$ at $\mathcal{P}_j$ is generated by some element of $C_{i_j}^{a_j}$.
\end{corollary}

\begin{proof}
For each index $i \in \{1,\dots,r\}$, pick a prime divisor $\mathcal{P}'_i$ of $m_{t_i}(T) \cdot m_{1/t_i}(T)$ which is a good prime for $E/k(T)$ unitizing $t_i$ (such a prime may be found since, from the Tchebotarev density theorem, $m_{t_i}(T) \cdot m_{1/t_i}(T)$ classically has infinitely many distinct prime divisors). Assume that the primes $\mathcal{P}'_1, \dots, \mathcal{P}'_r, \mathcal{P}_1,\dots,\mathcal{P}_s$ are distinct.

Apply theorem \ref{Beckmann4} to the larger set of primes $\{\mathcal{P}_j \, / \, j \in \{1,\dots,s\}\} \cup \{\mathcal{P}'_i \, / \, i \in \{1,\dots,r\}\}$, each $\mathcal{P}_j$ given with the couple $(i_j,a_j)$ of the statement and each $\mathcal{P}'_i$ with the couple $(i,1)$. The conclusion on the primes $\mathcal{P}_1,\dots,\mathcal{P}_s$ is exactly part (2) of corollary \ref{gc} and, according to our g-complete hypothesis, that on the primes $\mathcal{P}'_1,\dots,\mathcal{P}'_r$ provides part (1). 

To obtain that $t_0$ can be any term in some arithmetic progression, we use the more precise conclusion of addendum \ref{Beckmann4}. It provides some $\theta \in k$ such that conditions (1) and (2) simultaneously hold at any $t_{0,u} = \theta + u \, (\prod_{l \in S} x_{\mathcal{P}_l}^{a_l+1} \cdot \prod_{l \in S'} x_{\mathcal{P}'_l}^{2}) \not \in \{t_1,\dots,t_r\}$, with $S'$ the set of all $i \in \{1,\dots,r\}$ such that $t_i \not= \infty$ and $u$ any element of $k$ satisfying $v_{\mathcal{P}_j}(u) \geq 0$ for each $j \in \{1,\dots,s\}$ and $v_{\mathcal{P}'_i}(u) \geq 0$ for each $i \in \{1,\dots,r\}$.
\end{proof}

\begin{remark} More generally, the proof shows that the conclusion of corollary \ref{gc} remains true if there exists some subset $I \subset \{1,\dots,r\}$ satisfying the following two conditions:

\vspace{0.5mm}

\noindent
(1) the set $\{C_i \, / \, i \in I \} \cup \, \{C_{i_j}^{a_j} \, / \, j=1,\dots,s\}$ is g-complete,

\vspace{0.25mm}

\noindent
(2) for each index $i \in I$, the polynomial $m_{t_i}(T) \cdot m_{1/t_i}(T)$ has infinitely many distinct prime divisors.

\vspace{0.5mm}

\noindent
In particular, we do not require the base field $k$ to be hilbertian.
\end{remark}

\subsection{Examples}
Fix $k=\mathbb{Q}$ (for simplicity) and let $G$ be a finite group. As a consequence of corollary \ref{Hilbert}, we obtain that

\vspace{1mm}

\noindent
$(**)$ {\it{there is a finite set $\mathcal{S}_{\rm{exc}}$ of primes such that, given a finite set $\mathcal{S}$ of primes, there are infinitely many linearly disjoint Galois extensions of $\mathbb{Q}$ of group $G$ which each ramifies at each prime of $\mathcal{S} \smallsetminus \mathcal{S}_{\rm{exc}}$}},

\vspace{1mm}

\noindent
provided that the following condition is satisfied:

\vspace{1mm}

\noindent
(H1/$\mathbb{Q}$) {\it{there is a Galois extension $E/\mathbb{Q}(T)$ of group $G$ with $E/\mathbb{Q}$ regular and at least one $\mathbb{Q}$-rational branch point}}\footnote{More generally, condition $(**)$ remains true if there is a Galois extension $E/\mathbb{Q}(T)$ of group $G$ with $E/\mathbb{Q}$ regular and such that all but finitely many primes are prime divisors of the polynomial $m_{\bf{\underline{t}}}(T) \cdot m_{1/\bf{\underline{t}}}(T)$ (introduced in \S2.3).}. 

\vspace{2mm}

Not all finite groups satisfy the inverse Galois theory condition (H1/$\mathbb{Q}$): for example, \cite[corollary 1.3]{DF90} shows that such a finite group should be of even order\footnote{This remains true if $\mathbb{Q}$ is replaced by any number field $k \subset \mathbb{R}$.}. But some do. We recall below several of them to which we then apply corollary \ref{Hilbert}.

\subsubsection{Symmetric groups} Let $n$, $m$, $q$, $r$ be positive integers such that $n\geq 3$, $1 \leq m \leq n$, $(m,n)=1$ and $q(n-m)-rn=1$. Denote the splitting field over $\mathbb{Q}(T)$ of the trinomial $X^n-T^rX^m+T^q$ by $E$. Then the extension $E/\mathbb{Q}$ is regular and the splitting extension $E/\mathbb{Q}(T)$ has Galois group $S_n$ and branch point set $\{0, \infty, m^m n^{-n} (n-m)^{n-m}\}$, with corresponding inertia groups generated by the disjoint product of an $m$-cycle and an $(n-m)$-cycle at $0$, an $n$-cycle at $\infty$ and a transposition at $m^m n^{-n} (n-m)^{n-m}$. See \cite[$\S$2.4]{Sch00}. 

As $S_n$ has trivial center, one easily shows that the bad primes for $E/\mathbb{Q}(T)$ are exactly the primes $\leq n$. Then corollary \ref{Sn} below immediately follows from corollary \ref{Hilbert} (and lemma \ref{spec}):

\begin{corollary} \label{Sn}
Given a positive integer $s$, fix $s$ distinct primes $p_1,\dots,p_s$ $>n$ and $s$ couples $(C_1,a_1), \dots, (C_s,a_s)$  where, for each $j \in \{1,\dots,s\}$,

\vspace{0.5mm}

\noindent
- $C_j$ is the conjugacy class in $S_n$ of all the $n$-cycles or of all the disjoint products of an $m$-cycle and an $(n-m)$-cycle or of all the transpositions,

\vspace{0.5mm}

\noindent
- $a_j$ is a positive integer.

\vspace{0.5mm}

\noindent
Then, for infinitely many distinct points $t_0 \in \mathbb{Q}$, the splitting extensions $E_{t_0}/\mathbb{Q}$ over $\mathbb{Q}$ of the trinomials $X^n-{t_0}^r X^m + {t_0}^q$ are linearly disjoint and each satisfies the following two conditions:

\vspace{0.5mm}

\noindent
{\rm{(1)}} ${\rm{Gal}}(E_{t_0}/\mathbb{Q}) = S_n$,

\vspace{0.5mm}

\noindent
{\rm{(2)}} for each index $j \in \{1,\dots,s\}$, the inertia group of $E_{t_0}/\mathbb{Q}$ at $p_j$ is generated by some element of $C_j^{a_j}$.
\end{corollary}

As the inertia canonical conjugacy class set of $E/\mathbb{Q}(T)$ is g-complete \cite[\S2.4]{Sch00}, one may use corollary \ref{gc} (instead of corollary \ref{Hilbert}) to obtain a more precise conclusion on points $t_0$ at which the conclusion holds (at the cost of dropping the linearly disjointness condition).

\subsubsection{The Monster and other groups} Let $G$ be a centerless finite group which occurs as the Galois group of a Galois extension $E/\mathbb{Q}(T)$ such that $E/\mathbb{Q}$ is regular and with branch point set $\{0,1,\infty\}$. It is easily checked that the set of bad primes for such an extension is exactly the set of prime divisors of the order of $G$.

From the {\it{rigidity method}}, several centerless finite groups are known to occur as the Galois group of such an extension of $\mathbb{Q}(T)$ (see {\it{e.g.}} \cite{Ser92} and \cite{MM99}). For example, applying corollary \ref{Hilbert} to that of group the Monster group M, branch point set $\{0,1,\infty\}$ and inertia canonical invariant $(2A,3B,29A)$ (according to the Atlas \cite{ATL} notation for conjugacy classes of finite groups) provides corollary \ref{Monstre} below:

\begin{corollary} \label{Monstre}
Given a positive integer $s$, fix $s$ distinct primes $p_1,\dots,p_s$ $\geq 73$ or in $\{37,43,53,61,67\}$ and $s$ couples $(C_1,a_1), \dots,(C_s,a_s)$ where, for each index $j \in \{1,\dots,s\}$,

\vspace{0.5mm}

\noindent
- $C_j$ is a conjugacy class of {\rm{M}} in $\{2A, 3B, 29A\}$,

\vspace{0.5mm}

\noindent
- $a_j$ is a positive integer.

\vspace{0.5mm}

\noindent
Then there exist infinitely many linearly disjoint Galois extensions of $\mathbb{Q}$ of group {\rm{M}} whose inertia group at $p_j$ is generated by some element of $C_j^{a_j}$ for each index $j \in \{1,\dots,s\}$.
\end{corollary}

\subsection{Conjoining theorem \ref{Beckmann4} and previous results}
Fix $k=\mathbb{Q}$ for simplicity. As already noted in remark \ref{ur}, theorem \ref{Beckmann4} also includes trivial ramification. Previous works, namely \cite{DG11} and \cite{DG12}, are concerned with this kind of conclusions: D\`ebes and Ghazi show that, for each finite group $G$, any Galois extension $E/\mathbb{Q}(T)$ of group $G$ such that $E/\mathbb{Q}$ is regular has specializations with the same group which each is unramified at any finitely many prescribed large enough primes and such that the associated Frobenius at each such prime is in any specified conjugacy class of $G$. 

As stated in theorem \ref{DGL} below, it is in fact possible to conjoin these two results to obtain Galois extensions of $\mathbb{Q}$ of various finite groups with specified local behavior at finitely many given primes.

\subsubsection{Statement of the result} Let $G$ be a finite group and $E/\mathbb{Q}(T)$ a Galois extension of group $G$ with $E/\mathbb{Q}$ regular. Denote its branch point set by $\{t_1,\dots,t_r\}$ and its inertia canonical invariant by $(C_1,\dots,C_r)$.

Let ${\mathcal{S}}_{\rm{ra}}$ and ${\mathcal{S}}_{\rm{ur}}$ be two disjoint finite sets of good\footnote{Condition (4) of definition \ref{bon premier} may be removed for primes in $\mathcal{S}_{\rm{ur}}$.} primes for $E/\mathbb{Q}(T)$ such that $\mathcal{S}_{\rm{ur}} \not= \emptyset$ and each prime $p \in \mathcal{S}_{\rm{ur}}$ satisfies $p \geq r^2|G|^2$ \footnote{In \cite{DG12}, the bound is $p \geq 4r^2|G|^2$. This slight difference comes from a slight technical improvement in the bounds obtained from the Lang-Weil estimates (see \cite[$\S$3.2]{DL13}).}. For each prime $p \in \mathcal{S}_{\rm{ur}}$, fix a conjugacy class $C_p$ of $G$. For each prime $p \in \mathcal{S}_{\rm{ra}}$, let $a_p$ be a positive integer and $i_p \in \{1,\dots,r\}$ such that $t_{i_p} \not= \infty$, $p$ unitizes $t_{i_p}$ and is a prime divisor of $m_{t_{i_p}}(T) \cdot m_{1/t_{i_p}}(T)$.

Assume in theorem \ref{DGL} below that the set $\{C_{i_p}^{a_p} \, / \, p \in \mathcal{S}_{\rm{ra}}\} \cup \{C_{p} \, / \, p \in \mathcal{S}_{\rm{ur}}\}$ is g-complete. At the cost of throwing in more primes in $\mathcal{S}_{\rm{ur}}$ with appropriate associated conjugacy classes of $G$, we may assume that this hypothesis holds.

\begin{theorem} \label{DGL}There exists some integer $\theta$ satisfying the following conclusion. For each integer $t_0 \equiv \theta \, \, \, {\rm{mod}} \, \, \, (\prod_{p \in \mathcal{S}_{\rm{ur}}} p \cdot \prod_{p \in \mathcal{S}_{\rm{ra}}} p^{a_p+1})$, $t_0$ is not a branch point and the specialization $E_{t_0}/\mathbb{Q}$ of $E/\mathbb{Q}(T)$ at $t_0$ satisfies the following three conditions:

\vspace{0.5mm}

\noindent
{\rm{(1)}} ${\rm{Gal}}(E_{t_0}/\mathbb{Q})=G$,

\vspace{0.5mm}

\noindent
{\rm{(2)}} for each prime $p \in \mathcal{S}_{\rm{ra}}$, the inertia group of $E_{t_0}/\mathbb{Q}$ at $p$ is generated by some element of $C_{i_p}^{a_p}$,

\noindent
{\rm{(3)}} for each prime $p \in \mathcal{S}_{\rm{ur}}$, $p$ does not ramify in $E_{t_0}/\mathbb{Q}$ and the associated Frobenius is in the conjugacy class $C_{p}$.
\end{theorem}

\begin{remark} \label{remark DGL}
(1) The condition in the data requiring finitely many primes to be left aside cannot be removed in general. Otherwise, given a prime $p$, either one of conditions (2) and (3) provides specializations of $E/\mathbb{Q}(T)$ which each does not ramify at $p$. As a consequence of results of Plans and Vila, this last conclusion does not hold in general \cite[propositions 2.3 and 2.5]{PV05}.

\vspace{1.5mm}

\noindent
(2) The Specialization Inertia Theorem provides some limitations to the natural question of prescribing the decomposition group at each prime $p \in \mathcal{S}_{\rm{ra}}$ of the specialization $E_{t_0}/\mathbb{Q}$. Indeed, if a given solvable subgroup $H \subset G$ is the decomposition group at a given large enough prime $p$ of some specialization of $E/\mathbb{Q}(T)$ ramifying at $p$, then $H$ should contain some non trivial power of some element of some inertia canonical conjugacy class (and so the order of $H$ is not relatively prime to the product of the ramification indices of the branch points). Here is an example where this condition does not hold.

Fix an odd prime $p'$ and an integer $n$ such that $n \geq p'^2$ and $p'$ does not divide $n(n-1)$. Next pick a prime $p$ such that $p > n$ and $p \equiv 1 \, \, {\rm{mod}} \, \, p'$. Now consider the  Galois extension $E/\mathbb{Q}(T)$ of group $G=S_n$ with $E/\mathbb{Q}$ regular and branch point set $\{0,1,\infty\}$ provided by the rigid triple of conjugacy classes of $S_n$ given by that of all the $n$-cycles, that of all the $(n-1)$-cycles and that of all the tranpositions. Next fix a Galois extension $F_p/\mathbb{Q}_p$ of group $H=\mathbb{Z}/p'\mathbb{Z} \times \mathbb{Z}/p'\mathbb{Z} \subset G=S_n$ (as $n \geq p'^2)$; such an extension exists as $p \equiv 1 \, \, {\rm{mod}} \, \, p'$. Then $F_p/\mathbb{Q}_p$ is not a specialization of $E\mathbb{Q}_p/\mathbb{Q}_p(T)$. Indeed, as $p > n$ and 0, 1 and $\infty$ are the branch points, part (2) of the Specialization Inertia Theorem shows that the ramification index of the valuation ideal $p \mathbb{Z}_p$ in any specialization of $E\mathbb{Q}_p/\mathbb{Q}_p(T)$ is a divisor of $2n(n-1)$. As $p > n$, the ramification index of $p \mathbb{Z}_p$ in $F_p/\mathbb{Q}_p$ is equal to $p'$ and our claim follows.
\end{remark}

\subsubsection{Proof of theorem \ref{DGL}} We first recall how \cite{DG12} handles condition (3). Fix a prime $p \in \mathcal{S}_{\rm{ur}}$ and an element $g_p \in C_p$. Denote the order of $g_p$ by $e_p$. Let $F_p/\mathbb{Q}_p$ be the unique unramified Galois extension of $\mathbb{Q}_p$ of degree $e_p$, given together with an isomorphism $f:{\rm{Gal}}(F_p/\mathbb{Q}_p) \rightarrow \, \langle g_p \rangle$ satisfying $f(\sigma)=g_p$ with $\sigma$ the Frobenius of the extension $F_p/\mathbb{Q}_p$. Let $\varphi: {\rm{G}}_{\mathbb{Q}_p} \rightarrow \,  \langle g_p \rangle$ be the corresponding epimorphism. Since $p \geq r^2|G|^2$ and $p$ is a good prime for $E/\mathbb{Q}(T)$, \cite{DG12} provides some integer $\theta_p$ such that, for each integer $t \equiv \theta_p \, \, \, {\rm{mod}} \, \, \, p$, $t$ is not a branch point and the specialization $(E\mathbb{Q}_p)_t/\mathbb{Q}_p$ corresponds to the epimorphism $\varphi$.

For each prime $p \in {\mathcal{S}}_{\rm{ra}}$, addendum \ref{Beckmann4} provides some integer ${\theta'_p}$ such that, for every integer $t$ satisfying $t \equiv {\theta'_p} \, \,  {\rm{mod}} \, \, p^{a_p+1}$ and $t \not \in \{t_1,\dots,t_r\}$, the inertia group at $p$ of the specialization $E_{t}/\mathbb{Q}$ is generated by some element of $C_{i_p}^{a_p}$.

Next use the chinese remainder theorem to find some integer $\theta$ satisfying $\theta \equiv \theta_p \,  \, {\rm{mod}} \, \, p$ for each prime $p \in \mathcal{S}_{\rm{ur}}$ and $\theta\equiv {\theta'_p} \, \,   {\rm{mod}} \, \, p^{a_p+1}$ for each prime $p \in \mathcal{S}_{\rm{ra}}$. Then, for every integer $t_0$ such that $t_0 \equiv \theta \,  \, {\rm{mod}} \, \, (\prod_{p \in \mathcal{S}_{\rm{ur}}} p \cdot \prod_{p \in \mathcal{S}_{\rm{ra}}} p^{a_p+1})$, $t_0$ is not a branch point and the specialization $E_{t_0}/\mathbb{Q}$ of $E/\mathbb{Q}(T)$ at $t_0$ satisfies conditions (2) and (3).

Finally, for such a specialization point $t_0$, one has ${\rm{Gal}}(E_{t_0}/\mathbb{Q})=G$ according to our g-complete hypothesis, thus ending the proof.

\subsection{Proof of theorem \ref{Beckmann4}} 
We first show theorem \ref{Beckmann4} under the extra assumption that the set $S$ from addendum \ref{Beckmann4} satisfies $S=\{1,\dots,s\}$ ($\S$3.4.1) and next consider the case $S \not=\{1,\dots,s\}$ ($\S$3.4.2). 

For simplicity, denote in this subsection the irreducible polynomials over $k$ of $t_{i_1}, \dots,t_{i_s}$ (resp. of $1/t_{i_1}, \dots, 1/t_{i_s}$) by $m_{i_1}(T), \dots,m_{i_s}(T)$ (resp. by $m_{i_1}^*(T), \dots,m_{i_s}^*(T)$) respectively.

\subsubsection{First case: ${S}=\{1,\dots,s\}$} The main part of the proof consists in showing that there is an element $\theta \in A$ (not depending on $j$) such that $v_{\mathcal{P}_j}(m_{i_j}(\theta)) = a_j$ for each index $j \in \{1,\dots,s\}$. Indeed, for such a $\theta$, fix $u \in \bigcap_{l=1}^s A_{\mathcal{P}_l}$ such that $t_{0,u}= \theta + u \, \prod_{l=1}^s x_{\mathcal{P}_l}^{a_l+1}$ is not a branch point. For each index $j \in \{1,\dots,s\}$, one then has $v_{\mathcal{P}_j}(m_{i_j}(t_{0,u})) = a_j$ (remark \ref{rem div}), {\it{i.e.}} $I_{\mathcal{P}_j}(t_{0,u}, t_{i_j})=a_j$. Next apply part (1) of lemma \ref{lemme} and part (2) of the Specialization Inertia Theorem to conclude.

According to our assumptions, $\mathcal{P}_j$ is a prime divisor of $m_{i_j}(T)$ or of $m_{i_j}^*(T)$ for each index $j \in \{1,\dots,s\}$. In fact, from lemma \ref{drop} below, one may drop the polynomials $m_{i_1}^*(T), \dots, m_{i_s}^*(T)$.

\begin{lemma} \label{drop}
For each $j \in \{1,\dots,s\}$, $\mathcal{P}_j$ is a prime divisor of $m_{i_j}(T)$.
\end{lemma}

\begin{proof}
Indeed, if $\mathcal{P}_j$ is a prime divisor of $m_{i_j}^*(T)$ for some index $j$, then there exists some element $t \in A_{\mathcal{P}_j}$ such that $m_{i_j}^*(t) \in \mathcal{P}_j A_{\mathcal{P}_j}$. In particular, one has $t_{i_j} \not=0$ (otherwise $1 = m_{i_j}^*(t) \in \mathcal{P}_jA_{\mathcal{P}_j}$). Since $\mathcal{P}_j$ unitizes $t_{i_j}$, the constant coefficient $a_0$ of $m_{i_j}(T)$ satisfies $v_{\mathcal{P}_j}(a_0)=0$ and, from $t_{i_j} \not= \infty$, one has $t \not \in \mathcal{P}_j A_{\mathcal{P}_j}$. Hence, from $m_{i_j}^*(t)=(1/a_0) \, t^n \, m_{i_j}(1/t)$ (with $n={\rm{deg}}(m_{i_j}(T))$), one has $m_{i_j}(1/t) \in \mathcal{P}_j A_{\mathcal{P}_j}$, {\it{i.e.}} $\mathcal{P}_j$ is a prime divisor of $m_{i_j}(T)$.
\end{proof}

\begin{remark} \label{drop2}
In particular, lemma \ref{drop} shows that, if $\infty$ is not a branch point, then the two polynomials $m_{\bf{\underline{t}}}(T)$ and $m_{\bf{\underline{t}}}(T) \cdot m_{1/\bf{\underline{t}}}(T)$ have the same prime divisors (up to finitely many).
\end{remark}

For each index $j \in \{1,\dots,s\}$, pick an element $\theta_j \in A_{\mathcal{P}_j}$ such that $v_{\mathcal{P}_j}(m_{i_j}(\theta_j))>0$. The core of the construction consists in replacing the $s$-tuple $(\theta_1,\dots,\theta_s)$ by some suitable $s$-tuple $(\theta'_1,\dots,\theta'_s)$ satisfying $v_{\mathcal{P}_j}(m_{i_j}(\theta'_{j}))=a_j$ for each index $j \in \{1,\dots,s\}$.

\begin{lemma} \label{Taylor}
Let $j \in \{1,\dots,s\}$ and $d$ be a positive integer. Then there exists an element $\theta_{j,d} \in A_{\mathcal{P}_j}$ such that $v_{\mathcal{P}_j}(m_{i_j}(\theta_{j,d}))=d$.
\end{lemma}

\begin{proof}
We show lemma \ref{Taylor} by induction. If $v_{\mathcal{P}_j}(m_{i_j}(\theta_j))=1$, one can obviously take $\theta_{j,1} = \theta_j$. Otherwise, as noted in the last paragraph of the proof of proposition \ref{transition}, one can take $\theta_{j,1} = \theta_j + x_{\mathcal{P}_j} \in A_{\mathcal{P}_j}$.

We now explain how to produce an element $\theta_{j,2} \in A_{\mathcal{P}_j}$. From lemma \ref{derivee}, one has $v_{\mathcal{P}_j}(m'_{i_j}(\theta_{j,1}))=0$ and then $m'_{i_j}(\theta_{j,1}) \not =0$. First assume that $(1/2)m''_{i_j}(\theta_{j,1}) \in A_{\mathcal{P}_j} \smallsetminus \mathcal{P}_j A_{\mathcal{P}_j}$ and set $u= -(m_{i_j}(\theta_{j,1})/m'_{i_j}(\theta_{j,1})) + x^{3}_{\mathcal{P}_j}$. Taylor's formula yields
$$ m_{i_j}(\theta_{j,1} + u)=x_{\mathcal{P}_j}^{3} m'_{i_j}(\theta_{j,1}) + (1/2)u^2 m''_{i_j}(\theta_{j,1}) + u^3 R_j$$
with $R_j \in A_{\mathcal{P}_j}$. Hence one can take $\theta_{j,2} = \theta_{j,1} \, + \, u$ (this is an element of $A_{\mathcal{P}_j}$ since $v_{\mathcal{P}_j}(u)=1$) since one has $v_{\mathcal{P}_j}(x_{\mathcal{P}_j}^{3} m'_{i_j}(\theta_{j,1}))= 3$, $v_{\mathcal{P}_j}((1/2)u^2 m''_{i_j}(\theta_{j,1}))=2$ and $v_{\mathcal{P}_j}(u^3 R_j) \geq 3$. Now assume that  $v_{\mathcal{P}_j}((1/2)m''_{i_j}(\theta_{j,1})) \geq 1$ and set $\widetilde{u}= -(m_{i_j}(\theta_{j,1})/m'_{i_j}(\theta_{j,1}))$ $+ \, x^{2}_{\mathcal{P}_j}$. Taylor's formula yields $$ m_{i_j}(\theta_{j,1} + \widetilde{u})=x_{\mathcal{P}_j}^{2} m'_{i_j}(\theta_{j,1}) + (1/2)\widetilde{u}^2 m''_{i_j}(\theta_{j,1}) + \widetilde{u}^3 R_j$$
with $R_j \in A_{\mathcal{P}_j}$. Then one can take $\theta_{j,2} = \theta_{j,1} \, + \, \widetilde{u}$ (this is an element of $A_{\mathcal{P}_j}$ since $v_{\mathcal{P}_j}(\widetilde{u})=1$) since one has $v_{\mathcal{P}_j}(x_{\mathcal{P}_j}^{2} m'_{i_j}(\theta_{j,1}))= 2$, $v_{\mathcal{P}_j}((1/2)\widetilde{u}^2 m''_{i_j}(\theta_{j,1})) \geq 3$ and $v_{\mathcal{P}_j}(\widetilde{u}^3 R_j) \geq 3$.

Now fix an integer $d \geq 2$ and assume that an element $\theta_{j,d} \in A_{\mathcal{P}_j}$ has been constructed. We produce below an element $\theta_{j,d+1} \in A_{\mathcal{P}_j}$. As before, one has $v_{\mathcal{P}_j}(m'_{i_j}(\theta_{j,d}))=0$ and then $m'_{i_j}(\theta_{j,d}) \not =0$. Set $u= -(m_{i_j}(\theta_{j,d})/m'_{i_j}(\theta_{j,d})) + x^{d+1}_{\mathcal{P}_j}$. Taylor's formula yields
$$ m_{i_j}(\theta_{j,d} + u)=x_{\mathcal{P}_j}^{d+1} m'_{i_j}(\theta_{j,d}) + u^2 R_j$$
with $R_j \in A_{\mathcal{P}_j}$. Then one can take $\theta_{j,d+1} = \theta_{j,d}+u$ (this is an element of $A_{\mathcal{P}_j}$ since $v_{\mathcal{P}_j}(u)=d$) since one has $v_{\mathcal{P}_j}(x_{\mathcal{P}_j}^{d+1} m'_{i_j}(\theta_{j,d}))= d+1$ and $v_{\mathcal{P}_j}(u^2 R_j) \geq 2d > d+1$ ($d \geq 2$).
\end{proof}

For each $j \in \{1,\dots,s\}$, fix $\theta'_j \in A_{\mathcal{P}_j}$ such that $v_{\mathcal{P}_j}(m_{i_j}(\theta'_j))=a_j$. From the chinese remainder theorem, there are infinitely many distinct $\theta \in A$ such that $\theta - \theta'_j \in \mathcal{P}_j^{a_j+1} A_{\mathcal{P}_j}$ for each $j \in \{1,\dots,s\}$. Hence, for such a $\theta$, remark \ref{rem div} shows that one has $v_{\mathcal{P}_j}(m_{i_j}(\theta))=a_j$ for each $j \in \{1,\dots,s\}$, thus ending the proof in the case $S=\{1,\dots,s\}$.

\subsubsection{Second case: $S \not= \{1,\dots,s\}$} The proof of lemma \ref{drop} shows that the prime $\mathcal{P}_j$ is a prime divisor of $m_{i_j}(T)$ for each index $j \in S$. Next use lemma \ref{Taylor} to pick a $|S|$-tuple $(\theta_j)_{j \in S} \in \prod_{j \in S} A_{\mathcal{P}_j}$ satisfying $v_{\mathcal{P}_j}(m_{i_j}(\theta_j))=a_j$ for each index $j \in S$. Let $S^*=\{1,\dots,s\} \smallsetminus S$, {\it{i.e.}} $S^*$ is the set of all the indices $j \in \{1,\dots,s\}$ such that $t_{i_j} = \infty$. For each index $j \in S^*$, denote $x_{\mathcal{P}_j}^{a_j}$ by $\theta^*_j$. 

From the Artin-Whaples theorem ({\it{e.g.}} \cite[chapter XII, theorem 1.2]{Lan02}), there exists some element $\theta \in k$ satisfying these two conditions:

\vspace{0.5mm}

\noindent
(i) $v_{\mathcal{P}_j} (\theta - \theta_j) \geq a_j+1$ (and so $v_{\mathcal{P}_j}(\theta) \geq 0$) for each index $j \in {S}$,

\vspace{0.5mm}

\noindent
(ii) $v_{\mathcal{P}_j} (\theta - (1 /\theta^*_j)) \geq a_j+1$ (and so $v_{\mathcal{P}_j}(\theta) < 0$) for each index $j \in {S}^*$.

\vspace{0.5mm}

\noindent
Fix an element $u \in \bigcap_{l=1}^s A_{\mathcal{P}_l}$ such that $t_{0,u} = \theta + u \, \prod_{l \in {S}} x_{\mathcal{P}_l}^{a_l+1}$ is not a branch point. We show below that $I_{\mathcal{P}_j}(t_{0,u},t_{i_j})=a_j$ for each index $j \in \{1,\dots,s\}$. As in $\S$3.4.1, it then remains to apply part (1) of lemma \ref{lemme} and part (2) of the Specialization Inertia Theorem to conclude.

Let $j \in {S}$. Since $v_{\mathcal{P}_j} (t_{0,u}) \geq 0$, one has $I_{\mathcal{P}_j}(t_{0,u},t_{i_j}) = v_{\mathcal{P}_j}(m_{i_j}(t_{0,u}))$ and, as in the case ${S} = \{1,\dots,s\}$, one has $v_{\mathcal{P}_j}(m_{i_j}(t_{0,u}))=a_j$.

Let $j \in {S}^*$. Since $t_{i_j} = \infty$ and $v_{\mathcal{P}_j}(t_{0,u}) = v_{\mathcal{P}_j}(\theta) < 0$, one has $I_{\mathcal{P}_j}(t_{0,u},t_{i_j})= v_{\mathcal{P}_j}(1/t_{0,u}) = v_{\mathcal{P}_j}(1/\theta)$. But $v_{\mathcal{P}_j}(\theta^*_j) = a_j$ and $v_{\mathcal{P}_j}((1/\theta) - \theta^*_j) = v_{\mathcal{P}_j}((1/\theta^*_j) - \theta) - v_{\mathcal{P}_j}(\theta) + v_{\mathcal{P}_j}(\theta^*_j) \geq a_j+1$. Hence $v_{\mathcal{P}_j}(1/\theta)=a_j$.

\section{Non parametric extensions over number fields}

This section is devoted to non parametric extensions. We first recall some basics in $\S$4.1. Section 4.2 is devoted to theorem \ref{methode} which is our main result to give examples of such extensions. We next apply it to some classical Galois extensions in $\S$4.3.

\subsection{Basics} Let $k$ be a field and $G$ a finite group.

\begin{definition} \label{ext st para}
Let $E/k(T)$ be a Galois extension of group $G$ with $E/k$ regular and $\{t_1,\dots,t_r\}$ its branch point set. We say that $E/k(T)$ is {\it{G-parametric over $k$}} if any Galois extension $F/k$ of group $G$ occurs as the specialization $E_{t_0}/k$ of $E/k(T)$ at $t_0$ for some $t_0 \in \mathbb{P}^1(k) \smallsetminus \{t_1,\dots,t_r\}$.
\end{definition}

The more general property for which the condition is required for any Galois extension $F/k$ of group a given subgroup $H \subset G$ is studied in \cite{Leg15}.

\vspace{3mm}

We briefly recall some known examples of $G$-parametric and non $G$-parametric extensions in the case $k=\mathbb{Q}$.

\vspace{2.5mm}

\noindent
(1) {\it{Positive examples}}. Let $G$ be a finite group. If $G$ is one of the groups $\{1\}$, $\mathbb{Z}/2\mathbb{Z}$, $\mathbb{Z}/3\mathbb{Z}$ and $S_3$, there exists at least one $G$-parametric extension over $\mathbb{Q}$. This comes from (lemma \ref{spec} and) the fact that these four groups (are the only ones to) have a {\it{one parameter generic polynomial over $\mathbb{Q}$}}, {\it{i.e.}} a monic (with respect to $X$) separable polynomial $P(T,X) \in \mathbb{Q}[T][X]$ of group $G$ such that, for any extension $L/\mathbb{Q}$, any Galois extension $F/L$ of group $G$ occurs as the splitting extension over $L$ of some separable polynomial $P(t_0,X)$ with $t_0 \in L$ \cite[page 194]{JLY02}. Here are some examples of $G$-parametric extensions over $\mathbb{Q}$ which are in fact provided by one parameter generic polynomials over $\mathbb{Q}$:

\vspace{0.5mm}

\noindent
(a) the extension $\mathbb{Q}(T)/\mathbb{Q}(T)$ is $\{1\}$-parametric over $\mathbb{Q}$,

\vspace{0.5mm}

\noindent
(b) the extension $\mathbb{Q}(\sqrt{T})/\mathbb{Q}(T)$ is $\mathbb{Z}/2\mathbb{Z}$-parametric over $\mathbb{Q}$,

\vspace{0.5mm}

\noindent
(c) the splitting extension over $\mathbb{Q}(T)$ of the polynomial $X^3-TX^2+(T-3)X+1$ is $\mathbb{Z}/3\mathbb{Z}$-parametric over $\mathbb{Q}$ \cite[page 30]{JLY02},

\vspace{0.5mm}

\noindent
(d) the splitting extension over $\mathbb{Q}(T)$ of the trinomial $X^3+TX+T$ is $S_3$-parametric over $\mathbb{Q}$ \cite[page 30]{JLY02}.

\vspace{0.5mm}

\noindent
If $G$ is none of the previous four groups, it seems unknown whether there exists such an extension or not.

\vspace{2.5mm}

\noindent
(2) {\it{Negative examples.}} In addition to the example with $G=\mathbb{Z}/2\mathbb{Z}$ from the introduction, only a few negative examples are known:

\vspace{0.5mm}

\noindent
(a) {\hbox{no Galois extension $E/\mathbb{Q}(T)$ of group $S_7$ with $E/\mathbb{Q}$ regular and}}
 branch point set $\{0, 1, \infty\}$ is $S_7$-parametric over $\mathbb{Q}$ \cite[example 1.1]{Bec94},

\vspace{0.75mm}

\noindent
(b) for every finite group $G \not= (\mathbb{Z}/2\mathbb{Z})^2$, $S_3$, $D_4$, $D_6$ which occurs as the Galois group of a totally real Galois extension of $\mathbb{Q}$, no Galois extension $E/\mathbb{Q}(T)$ of group $G$ with $E/\mathbb{Q}$ regular and three branch points is $G$-parametric over $\mathbb{Q}$ \cite[proposition 1.2]{DF90}.

\subsection{Criteria for non parametricity} Let $k$ be a number field, $A$ the integral closure of $\mathbb{Z}$ in $k$, $G$ a finite group and $E_1/k(T)$, $E_2/k(T)$ two Galois extensions of group $G$ with $E_1/k$ and $E_2/k$ regular. For each index $i \in \{1,2\}$, denote the product of the two polynomials introduced in $\S$2.3 from the branch points of $E_i/k(T)$ by $m_{{\bf{\underline{t}}},i}(T) \cdot m_{1/{\bf{\underline{t}}},i}(T)$ and the inertia canonical invariant of $E_i/k(T)$ by $(C_{1,i},\dots,C_{r_i,i})$.

\subsubsection{Statement of the result} Consider the following two conditions:

\vspace{1mm}

\noindent
(Branch Point Hypothesis) {\it{there exist infinitely many distinct primes of $A$ which each is a prime divisor of $m_{{\bf{\underline{t}}},1}(T)\cdot m_{1/{\bf{\underline{t}}},1}(T)$ but not of $m_{{\bf{\underline{t}}},2}(T)\cdot m_{1/{\bf{\underline{t}}},2}(T)$}},

\vspace{1mm}

\noindent
(Inertia Hypothesis) {\it{$\{C_{1,1},\dots,C_{r_1,1}\} \not \subset \{C_{1,2}^{a}, \dots, C_{r_2,2}^{a} \,  / \, a \in \mathbb{N}\}$.}}

\begin{theorem} \label{methode}
Under either one of these two conditions, the following non $G$-parametricity condition holds:

\vspace{1mm}

\noindent
{\rm{(non $G$-parametricity)}} there are infinitely many linearly disjoint Galois extensions of $k$ of group $G$ which are not specializations of $E_2/k(T)$.

\vspace{1mm}

\noindent
In particular, the extension $E_2/k(T)$ is not $G$-parametric over $k$. Moreover these Galois extensions of $k$ of group $G$ may be obtained by specializing $E_1/k(T)$.
\end{theorem}

\begin{remark} \label{rem geo}
As the Inertia Hypothesis does not depend on the base field $k$ \footnote{This is obviously false for the Branch Point Hypothesis.}, one obtains, under the Inertia Hypothesis, the following geometric non $G$-parametricity condition\footnote{As in theorem \ref{methode}, one may add that the Galois extensions of group $G$ whose existence is claimed may be obtained by specialization.}:

\vspace{1mm}

\noindent
{\rm{(geometric non $G$-parametricity)}} {\it{for any number field $k'$ containing $k$, there are infinitely many linearly disjoint Galois extensions of $k'$ of group $G$ which each is not a specialization of $E_2k'/k'(T)$.}}
\end{remark}

For simplicity, we have only considered the number field case. At the cost of some technical adjustments, similar criteria also hold for more general base fields. This is explained in \cite{Leg15}.

\subsubsection{Proof of theorem \ref{methode}} Denote the branch point set of $E_1/k(T)$ by $\{t_{1,1},\dots,t_{r_1,1}\}$ and, for each index $l \in \{1,\dots,r_1\}$, the irreducible polynomial of $t_{l,1}$ (resp. of $1/t_{l,1}$) over $k$ by $m_{l,1}(T)$ (resp. by $m_{l,1}^*(T)$).

First assume that the Branch Point Hypothesis holds. Then there exists some index $l \in \{1,\dots,r_1\}$ such that the polynomial $m_{l,1}(T) \cdot m_{l,1}^*(T)$ has infinitely many distinct prime divisors $\mathcal{P}$ which each is not a prime divisor of $m_{{\bf{\underline{t}}},2}(T)\cdot m_{1/{\bf{\underline{t}}},2}(T)$. Furthermore, up to excluding finitely many of these primes, one may also assume that such a prime $\mathcal{P}$ satisfies the following two conditions:

\vspace{0.5mm}

\noindent
(i) $\mathcal{P}$ is a good prime for $E_1/k(T)$ unitizing $t_{l,1}$,

\vspace{0.5mm}

\noindent
(ii) $\mathcal{P}$ is a good prime for $E_2/k(T)$ unitizing each of its branch points.

\vspace{0.5mm}

\noindent
For such a prime $\mathcal{P}$, apply corollary \ref{Hilbert} to construct infinitely many linearly disjoint specializations $F_\mathcal{P}/k$ of $E_1/k(T)$ of group $G$ which each ramifies at $\mathcal{P}$. From corollary \ref{transition2}, such a specialization $F_\mathcal{P}/k$ is not a specialization of $E_2/k(T)$ and the conclusion follows.

Now assume that the Inertia Hypothesis holds. Fix an index $l \in \{1,\dots,r_1\}$ such that $C_{l,1}$ is not in the set $\{C_{1,2}^{a},\dots,C_{r_2,2}^a \,  /  \, a \in \mathbb{N}\}$. From the Tchebotarev density theorem, $m_{l,1}(T) \cdot m_{l,1}^*(T)$ has infinitely many distinct prime divisors $\mathcal{P}$ which may be assumed as before to further satisfy conditions (i) and (ii) above. For such a $\mathcal{P}$, apply corollary \ref{Hilbert} to construct infinitely many linearly disjoint specializations $F_\mathcal{P}/k$ of $E_1/k(T)$ of group $G$ whose inertia group at $\mathcal{P}$ is generated by some element of $C_{l,1}$. If such a specialization $F_\mathcal{P}/k$ is a specialization of $E_2/k(T)$, then, from the Specialization Inertia Theorem, there exist some index $j \in \{1,\dots,r_2\}$ and some positive integer $a$ such that the inertia group of $F_\mathcal{P}/k$ at $\mathcal{P}$ is generated by some element of $C_{j,2}^a$; a contradiction. Hence the conclusion follows.

\subsection{Examples} Given a number field $k$, we now use theorem \ref{methode} to give some new examples of non $G$-parametric extensions over $k$. Our method consists in comparing the ramification data (branch points and inertia canonical invariants) of two Galois extensions of $k(T)$ with the same group $G$. From theorem \ref{methode}, one obtains practical sufficient conditions related to the ramification data of a single Galois extension $E/k(T)$ of group $G$ with $E/k$ regular for this extension not to be $G$-parametric over $k$. Corollaries \ref{4 pts} and \ref{Sn gen} below are typical examples of this approach. We next apply these results to some classical Galois extensions of $k(T)$ of group $G$ ({\it{e.g.}} corollaries \ref{Z/2Z} and \ref{A5 2}).

\subsubsection{Galois extensions with four branch points} Let $G$ be a finite group and $k$ a number field. Assume that the following condition, which has already appeared in $\S$3.2 in the case $k=\mathbb{Q}$, is satisfied:

\vspace{1.5mm}

\noindent
(H1/$k$) {\it{there is a Galois extension $E/k(T)$ of group $G$ with $E/k$ regular and at least one $k$-rational branch point.}}

\vspace{1.5mm}

As already noted in $\S$3.2, not all finite groups satisfy condition (H1/$k$) for a given number field $k$. However every finite group satisfies condition (H1/$k$) for suitable number fields $k$. 

Indeed it classically follows from the Riemann existence theorem that, if $r$ is strictly bigger than the rank of $G$ and $t_1,\dots,t_r$ are $r$ distinct points in $\mathbb{P}^1(\overline{\mathbb{Q}}$), then there exists some Galois extension $\overline{E}/\overline{\mathbb{Q}}(T)$ of group $G$ and branch point set $\{t_1,\dots,t_r\}$ ({\it{e.g.}} \cite[$\S$12]{Deb01a}). Hence condition (H1/$k$) holds for every number field $k$ that is a field of definition of $\overline{E}/\overline{\mathbb{Q}}(T)$ and of at least one of its branch points.

\begin{corollary} \label{4 pts}
Let $E/k(T)$ be a Galois extension of group $G$ with $E/k$ regular and four branch points. Assume that none of them is $k$-rational. Then $E/k(T)$ satisfies the {\rm{(non $G$-parametricity)}} condition.
\end{corollary}

In the case $E/k(T)$ has at least one $k$-rational branch point, the proof below fails. However the conclusion of corollary \ref{4 pts} may still happen: we give in \cite{Leg15} a Galois extension $E/\mathbb{Q}(T)$ of group $S_3$ such that $E/\mathbb{Q}$ is regular and with two $\mathbb{Q}$-rational and two complex conjugate branch points which satisfies the {\rm{(non $S_3$-parametricity)}} condition.

\begin{proof}
The branch points of $E/k(T)$ lead to either only one orbit $O_1$ of cardinality 4 or two orbits $O_2$ and $O_3$ of cardinality 2 under the action of ${\rm{G}}_k$. For each index $i \in \{1,2,3\}$, pick $t_i \in O_i$. From remark \ref{drop2} and our assumption on the branch points, the polynomials $m_{t_1}(T)$ and $m_{\bf{\underline{t}}}(T) \cdot m_{1/\bf{\underline{t}}}(T)$ in the first situation, $m_{t_2}(T) \cdot m_{t_3}(T)$ and $m_{\bf{\underline{t}}}(T) \cdot m_{1/\bf{\underline{t}}}(T)$ in the second one, have the same prime divisors (up to finitely many). We show below that there exist infinitely many distinct primes of the integral closure $A$ of $\mathbb{Z}$ in $k$ which each is not a prime divisor of $m_{t_1}(T)$ (resp. of $m_{t_2}(T) \cdot m_{t_3}(T)$). With $E'/k(T)$ any Galois extension of group $G$ with $E'/k$ regular and at least one $k$-rational branch point (condition (H1/$k$)), this shows that $E'/k(T)$ and $E/k(T)$ satisfy the Branch Point Hypothesis. Theorem \ref{methode} provides the desired conclusion.

In the first case, our claim follows from the irreducibility of $m_{t_1}(T)$ over $k$ and the fact that ${\rm{deg}}(m_{t_1}(T)) \geq 2$ ({\it{e.g.}} \cite[theorem 9]{Hei67}). In the second one, first assume that $k(t_2)=k(t_3)$. Then $m_{t_2}(T)$ and $m_{t_3}(T)$ have the same prime divisors (up to finitely many). As in the first case, there exist infinitely many distinct primes which each is not a prime divisor of $m_{t_2}(T)$, and so not of $m_{t_2}(T) \cdot m_{t_3}(T)$ either. 

Now assume that $k(t_2) \not= k(t_3)$. For each index $i \in \{2,3\}$, let $\sigma_i \, : \, {\rm{Gal}}(k(t_i)/k) \rightarrow S_{2}$ be the action of ${\rm{Gal}}(k(t_i)/k)$ on the roots of $m_{t_i}(T)$. Then ${\rm{Gal}}(k(t_2,t_3)/k)$ is isomorphic to ${\rm{Gal}}(k(t_2)/k) \times {\rm{Gal}}(k(t_3)/k)$ and $\sigma_2 \times \sigma_3 : {\rm{Gal}}(k(t_2)/k) \times {\rm{Gal}}(k(t_3)/k) \rightarrow S_{4}$ corresponds to the action of ${\rm{Gal}}(k(t_2,t_3)/k)$ on the roots of $m_{t_2}(T) \cdot m_{t_3}(T)$. From the Tchebotarev density theorem, there exist infinitely many distinct primes of $A$ such that the associated Frobenius is conjugate in ${\rm{Gal}}(k(t_2,t_3)/k)$ to $(g_2,g_3)$ where, for each index $i \in \{2,3\}$, $g_i$ denotes the unique non trivial element of ${\rm{Gal}}(k(t_i)/k)$. Hence none of these primes is a prime divisor of $m_{t_2}(T) \cdot m_{t_3}(T)$, thus ending the proof.
\end{proof}

We now apply corollary \ref{4 pts} to some classical Galois extensions of $k(T)$. Our first example is devoted to quadratic extensions while our second one is concerned with a Galois extension of group the alternating group $A_5$ produced by Mestre \cite{Mes90}. 

First remark that condition (H1/$k$) holds for each of these two groups over any number field $k$ ({\it{e.g.}} \cite[proposition 7.4.1 and theorem 8.2.2]{Ser92} for $A_5$).

\vspace{2.5mm} 

\noindent
(a) Let $k$ be a number field, $P(T) \in k[T]$ a separable polynomial and $\{t_1,\dots,t_r\}$ its root set. One easily shows that $(1,\sqrt{P(T)})$ is a $\overline{\mathbb{Q}}[T]$-basis of the integral closure of $\overline{\mathbb{Q}}[T]$ in $\overline{\mathbb{Q}}(T)(\sqrt{P(T)})$. Hence the branch point set of the extension $k(T)(\sqrt{P(T)})/k(T)$ is either $\{t_1,\dots,t_r\}$ or $\{t_1,\dots,t_r\} \cup \{\infty\}$. Moreover its branch point number is even from the Riemann-Hurwitz formula. Then corollary \ref{Z/2Z} below follows:

\begin{corollary} \label{Z/2Z}
Assume that ${\rm{deg}}(P(T))=4$ and $P(T)$ has no root in $k$. Then the extension $k(T)(\sqrt{P(T)})/k(T)$ satisfies the {\rm{(non $\mathbb{Z}/2\mathbb{Z}$-parametricity)}} condition.
\end{corollary}

\noindent
(b) Let $k$ be a number field such that $\mathbb{Q}(i) \subset k$. From \cite{Mes90}, the splitting field over $k(T)$ of $P(T,X)= (X^5-X) - T \, (25X^4-9)$ provides a Galois extension $E/k(T)$ of group $A_5$ with $E/k$ regular. Its branch points are the roots of the polynomial $S(T)= 1 + (5^5 \cdot 3^3) \, T^{4}$. 

From $3 \, S(T/15)=3+5T^4 = (1/5) \, (5T^2-i \sqrt{15}) \, (5T^2+i \sqrt{15})$ (and the fact that $i \in k$), the following two conditions hold:

\noindent
- the polynomial $S(T)$ is irreducible over $k$ if and only if $\sqrt{15} \not  \in k$,

\noindent
- the polynomial $S(T)$ is the product of two quadratic irreducible polynomials over $k$ if and only if $5i\sqrt{15} \in k \smallsetminus k^2$.

\noindent
Then corollary \ref{A5 2} below follows:

\begin{corollary} \label{A5 2}
Assume that $\mathbb{Q}(i) \subset k$ and $5i\sqrt{15} \not \in k^2$. Then the extension $E/k(T)$ satisfies the {\rm{(non $A_5$-parametricity)}} condition.
\end{corollary}

\subsubsection{Regular realizations of symmetric groups} Let $n \geq 3$ be an integer. Recall that {\it{the type of a permutation}} $\sigma \in S_n$ is the (multiplicative) divisor of all lengths of disjoint cycles involved in the cycle decomposition of $\sigma$ (for example, an $n$-cycle is of type $n^1$). The conjugacy class in $S_n$ of elements of type $1^{l_1} \dots n^{l_n}$ is denoted by $[1^{l_1} \dots n^{l_n}]$.

Let $k$ be a number field and $E/k(T)$ a Galois extension of group $S_n$ with $E/k$ regular. Denote its inertia canonical invariant by $(C_1,\dots,C_r)$. 

\begin{corollary} \label{Sn gen}
Assume that the following condition holds:

\vspace{1mm}

\noindent
{\rm{(H2)}} one of the conjugacy classes $[n^1]$, $[m^1 (n-m)^1]$, where $m$ is any integer such that $1 \leq m \leq n$ and $(m,n)=1$, is not in $\{C_1,\dots,C_r\}$.

\vspace{1mm}

\noindent
Then $E/k(T)$ satisfies the {\rm{(geometric non $S_n$-parametricity)}} condition.
\end{corollary}

An example of Galois extension of $k(T)$ of group $S_n$ satisfying condition (H2) is given after the proof below.

\begin{proof}
First assume that $[n^1]$ is not in $\{C_1,\dots,C_r\}$. Then $[n^1]$ is not in $\{C_1^a,\dots,C_r^a \, / \, a \in \mathbb{N}\}$ either. With $E'/k(T)$ any Galois extension of group $S_n$ with $E/k$ regular and inertia canonical invariant $([n^1], [m^1(n-m)^1], [1^{n-2} 2^1])$ \cite[$\S$7.4.1 and theorem 8.1.1]{Ser92}, this shows that the two extensions $E'/k(T)$ and $E/k(T)$ satisfy the Inertia Hypothesis. Remark \ref{rem geo} provides the desired conclusion.

If $[m^1 (n-m)^1]$ is not in $\{C_1,\dots,C_r\}$ for some $m$ as in the statement, repeat the same argument with $[n^1]$ replaced by $[m^1 (n-m)^1]$.
\end{proof}

For example, the regular realization of $S_n$ used in the proof satisfies condition (H2) if $\varphi(n) \geq 3$ (with $\varphi$ the Euler function), {\it{i.e.}} if $n=5$ or $n \geq 7$. Other examples are given in \cite{Leg15}.

Moreover condition (H2) obviously holds if $r \leq \varphi(n)/2$. This and the Riemann existence theorem may be conjoined to give examples of non $S_n$-parametric extensions over suitable number fields if $\varphi(n) \geq 6$.

\bibliography{Biblio2}

\newcommand{\etalchar}[1]{$^{#1}$}
\begin{thebibliography}{NSW08}

\bibitem[Bec91]{Bec91}
Sybilla Beckmann.
\newblock On extensions of number fields obtained by specializing branched
  coverings.
\newblock {\em J. Reine Angew. Math.}, 419:27--53, 1991.

\bibitem[Bec94]{Bec94}
Sybilla Beckmann.
\newblock Is every extension of $\mathbb{Q}$ the specialization of a branched
  covering?
\newblock {\em J. Algebra}, 164(2):430--451, 1994.

\bibitem[C{\etalchar{+}}85]{ATL}
J~.H. Conway et~al.
\newblock {\em Atlas of finite groups. Maximal subgroups and ordinary
  characters for simple groups. With computational assistance from J. G.
  Thackray}.
\newblock Oxford University Press, Eynsham, 1985.

\bibitem[D{\`{e}}b99]{Deb99a}
Pierre D{\`{e}}bes.
\newblock Galois covers with prescribed fibers: the {B}eckmann-{B}lack problem.
\newblock {\em Ann. Scuola Norm. Sup. Pisa Cl. Sci. (4)}, 28(2):273--286, 1999.

\bibitem[D{\`e}b01]{Deb01a}
Pierre D{\`e}bes.
\newblock M\'ethodes topologiques et analytiques en th\'eorie inverse de
  {G}alois: th\'eor\`eme d'existence de {R}iemann.
\newblock In {\em Arithm\'etique de rev\^etements alg\'ebriques
  ({S}aint-{\'E}tienne, 2000)}, volume~5 of {\em S\'emin. Congr.}, pages
  27--41. Soc. Math. France, Paris, 2001.

\bibitem[D{\`e}b09]{Deb09}
Pierre D{\`e}bes.
\newblock {\em Arithm\'etique des rev\^etements de la droite}.
\newblock {L}ecture notes, 2009.
\newblock {A}t http://math.univ-lille1.fr/\~{}pde/ens.html.

\bibitem[DF90]{DF90}
Pierre D{\`e}bes and Michael~D. Fried.
\newblock Rigidity and real residue class fields.
\newblock {\em Acta Arith.}, 56(4):291--323, 1990.

\bibitem[DG11]{DG11}
Pierre D\`ebes and Nour Ghazi.
\newblock Specializations of {G}alois covers of the line.
\newblock In {\em ``{A}lexandru {M}yller" {M}athematical {S}eminar}, volume
  1329 of {\em AIP Conf. Proc.}, pages 98--108. Amer. Inst. Phys., Melville,
  NY, 2011.

\bibitem[DG12]{DG12}
Pierre D{\`e}bes and Nour Ghazi.
\newblock Galois covers and the {H}ilbert-{G}runwald property.
\newblock {\em Ann. Inst. Fourier (Grenoble)}, 62(3):989--1013, 2012.

\bibitem[DL12]{DL12}
Pierre D\`{e}bes and Fran\c{c}ois Legrand.
\newblock Twisted covers and specializations.
\newblock In {\em Galois-{T}eichmueller theory and {A}rithmetic {G}eometry},
  pages 141--162. Proceedings for {C}onferences in {K}yoto (October 2010), {H}.
  {N}akamura, {F}. {P}op, {L}. {S}chneps, {A}. {T}amagawa eds., {A}dvanced
  {S}tudies in {P}ure {M}athematics 63, 2012.

\bibitem[DL13]{DL13}
Pierre D\`ebes and Fran\c{c}ois Legrand.
\newblock Specialization results in {G}alois theory.
\newblock {\em Trans. Amer. Math. Soc.}, 365(10):5259--5275, 2013.

\bibitem[Flo02]{Flo02}
St\'ephane Flon.
\newblock {\em Mauvaises places ramifi\'ees dans le corps des modules d'un
  rev\^etement}.
\newblock PhD thesis, Universit\'e des {S}ciences et {T}echnologies de {L}ille,
  France, 2002.

\bibitem[Fri95]{Fri95}
Michael~D. Fried.
\newblock Introduction to modular towers: generalizing dihedral group--modular
  curve connections.
\newblock In {\em Recent developments in the inverse Galois problem (Seattle,
  WA, 1993)}, volume 186 of {\em Contemp. Math.}, pages 111--171. Amer. Math.
  Soc., Providence, RI, 1995.

\bibitem[Gey78]{Gey78}
Wulf-Dieter Geyer.
\newblock Galois groups of intersections of local fields.
\newblock {\em Israel J. Math.}, 30(4):382--396, 1978.

\bibitem[Hei67]{Hei67}
Hans~Arnold Heilbronn.
\newblock Zeta-functions and {L}-functions.
\newblock In {\em Algebraic {N}umber {T}heory ({P}roc. {I}nstructional {C}onf.,
  {B}righton, 1965)}, pages 204--230. Thompson, {W}ashington, {D}.{C}., 1967.

\bibitem[JLY02]{JLY02}
Christian~U. Jensen, Arne Ledet, and Noriko Yui.
\newblock {\em Generic polynomials. Constructive Aspects of the Inverse Galois
  Problem}.
\newblock Cambridge University Press, 2002.

\bibitem[Jor72]{Jor72}
Camille Jordan.
\newblock Recherches sur les substitutions.
\newblock {\em J. Liouville}, 17:351--367, 1872.

\bibitem[KM04]{KM04}
J\"urgen Kl\"{u}ners and Gunter Malle.
\newblock Counting nilpotent {G}alois extensions.
\newblock {\em J. Reine Angew. Math.}, 572:1--26, 2004.

\bibitem[Lan02]{Lan02}
Serge Lang.
\newblock {\em Algebra}, volume 211 of {\em Graduate {T}exts in {M}athematics}.
\newblock Springer-Verlag, New York, revised third edition, 2002.

\bibitem[Leg13]{Leg13c}
Fran\c{c}ois Legrand.
\newblock {\em Sp\'ecialisations de rev\^etements et th\'eorie inverse de
  {G}alois}.
\newblock PhD thesis, Universit\'e Lille 1, France, 2013.
\newblock At https://sites.google.com/site/francoislegranden/recherche.

\bibitem[Leg15]{Leg15}
Fran\c{c}ois Legrand.
\newblock Parametric {G}alois extensions.
\newblock {\em Journal of Algebra}, 422:187--222, 2015.

\bibitem[Mes90]{Mes90}
Jean-{F}ran\c{c}ois Mestre.
\newblock Extensions r\'eguli\`eres de $\mathbb{Q}({T})$ de groupe de {G}alois
  $\tilde{A}_n$.
\newblock {\em J. Algebra}, 131(2):483--495, 1990.

\bibitem[MM99]{MM99}
Gunter Malle and B.~Heinrich Matzat.
\newblock {\em Inverse Galois Theory}.
\newblock Springer Monographs in Mathematics. Springer-Verlag, Berlin, 1999.

\bibitem[Neu79]{Neu79}
J\"urgen Neukirch.
\newblock On solvable number fields.
\newblock {\em Invent. Math.}, 53(2):135--164, 1979.

\bibitem[NSW08]{NSW08}
J\"urgen Neukirch, Alexander Schmidt, and Kay Wingberg.
\newblock {\em Cohomology of Number Fields}, volume 323 of {\em Grundlehren de
  mathematischen Wissenschaften}.
\newblock Springer, Berlin, second edition, 2008.

\bibitem[PV05]{PV05}
Bernat Plans and N\'{u}ria Vila.
\newblock {G}alois covers of $\mathbb{P}^1$ over $\mathbb{Q}$ with prescribed
  local or global behavior by specialization.
\newblock {\em J. Th\'eor. Nombres Bordeaux}, 17(1):271--282, 2005.

\bibitem[Sch00]{Sch00}
Andrzej Schinzel.
\newblock {\em Polynomials with special regard to reducibility}, volume~77 of
  {\em Encyclopedia of Mathematics and its Applications}.
\newblock Cambridge University Press, Cambridge, 2000.

\bibitem[Ser92]{Ser92}
Jean-Pierre Serre.
\newblock {\em Topics in Galois Theory}, volume~1 of {\em Research Notes in
  Mathematics}.
\newblock Jones and Bartlett Publishers, Boston, MA, 1992.

\bibitem[Tra90]{Tra90}
Artur Travesa.
\newblock Nombre d'extensions ab\'eliennes sur $\mathbb{Q}$.
\newblock {\em S\'em. Th\'eor. Nombres Bordeaux (2)}, 2(2):413--423, 1990.

\end{thebibliography}
\bibliographystyle{alpha}

\end{document}